\documentclass[11pt, a4paper, final]{amsart}
\usepackage[T1]{fontenc}
\usepackage{lmodern}
\usepackage{mathtools}
\usepackage[utf8]{inputenc}

\usepackage{amsfonts}
\usepackage{amsmath}
\usepackage{amssymb}
\usepackage{amsthm}
\usepackage{mathtools}
\usepackage{paralist, tabularx}
\usepackage{enumitem}
\usepackage[utf8]{inputenc}
\usepackage{mathabx,epsfig}
\usepackage{tikz-cd}
\usepackage{verbatim}

\usepackage{etoolbox}

\usepackage{xcolor}
\definecolor{green}{RGB}{0,127,0}
\definecolor{redd}{RGB}{191,0,0}
\definecolor{red}{RGB}{105,89,205}
\usepackage[colorlinks=true]{hyperref}

\usepackage[notref, notcite]{showkeys}
\usepackage[cmtip,arrow]{xy}
\usepackage[backend=biber,
url=false,
isbn=false,
backref=true,
citestyle=alphabetic,
bibstyle=alphabetic,
autocite=inline,
maxnames=99,
minalphanames=4,
maxalphanames=4,
sorting=nyt,]{biblatex}
\DeclareMathOperator{\UT}{UT}
\DeclareMathOperator{\T}{T}
\DeclareMathOperator{\Stab}{Stab}

\newcommand{\R}{{\mathbb{R}}}
\newcommand{\Z}{{\mathbb{Z}}}
\newcommand{\N}{{\mathbb{N}}}

\newcommand{\G}{{\bar{G}}}

\newcommand{\RR}{{\bar{R}}}

\newcommand{\RRa}{(\RR,+)}
\newcommand{\RRaD}{(\RR,+)^{0}}
\newcommand{\RRaT}{(\RR,+)^{00}}
\newcommand{\RRaI}{(\RR,+)^{000}}
\newcommand{\RRiD}{\RR^{0}}
\newcommand{\RRiT}{\RR^{00}}
\newcommand{\RRiI}{\RR^{000}}

\newcommand{\nref}[2]{\hyperref[#1]{\ref*{#1}$_{#2}$}}

\DeclareMathOperator{\topo}{{top}}

\DeclareMathOperator{\lin}{{Lin}}
\DeclareMathOperator{\Th}{{Th}}

\newtheorem{theorem}{Theorem}
\numberwithin{theorem}{section}
\newtheorem{lemma}[theorem]{Lemma}

\newtheorem{fact}[theorem]{Fact}
\newtheorem{proposition}[theorem]{Proposition}

\newtheorem{question}[theorem]{Question}
\newtheorem{corollary}[theorem]{Corollary}

\newtheorem{clm}{Claim}
\newtheorem*{clm*}{Claim}

\theoremstyle{definition}
\newtheorem{definition}[theorem]{Definition}

\newtheorem{example}[theorem]{Example}

\theoremstyle{remark}
\newtheorem{remark}[theorem]{Remark}

\AtEndEnvironment{proof}{\setcounter{clm}{0}}
\newenvironment{clmproof}[1][\proofname]{\proof[#1]}{\endproof}

\newcommand{\orcidlogo}{\includegraphics[height=\fontcharht\font`\B]{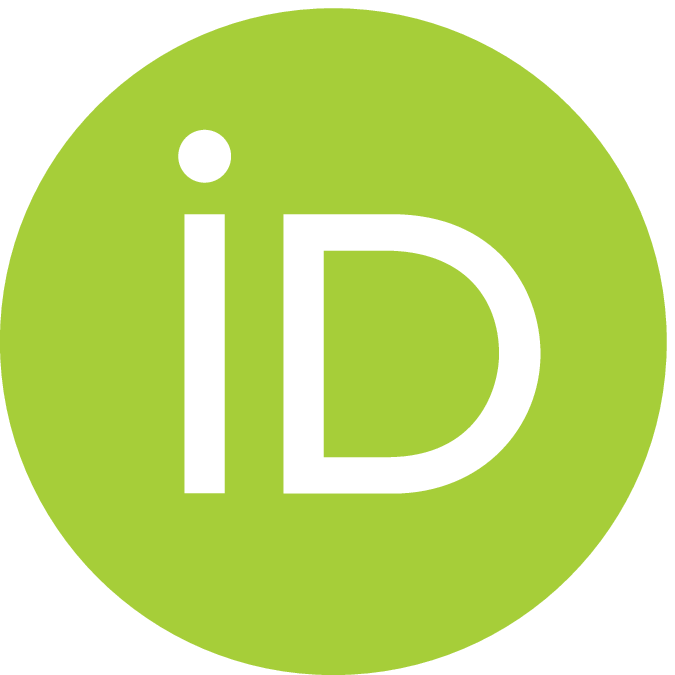}}
\newcommand{\orcid}[1]{\href{#1}{\orcidlogo #1}}

\title{Generating ideals by additive subgroups of rings}

\author{Krzysztof Krupi\'{n}ski}
\thanks{\noindent The first author is supported by the Narodowe Centrum Nauki grants no. 2016/22/E/ST1/00450 and 2018/31/B/ST1/00357.}
\author{Tomasz Rzepecki}
\thanks{The second author is supported by the Narodowe Centrum Nauki grants no. 2016/22/E/ST1/00450.}
\address{\parbox{\linewidth}{Instytut Matematyczny, Uniwersytet Wroc{\l}awski\\
pl. Grunwaldzki 2, 50-384 Wroc{\l}aw, Poland}}
\address{Krzysztof Krupi\'{n}ski \orcid{https://orcid.org/0000-0002-2243-4411}}
\email{Krzysztof.Krupinski@math.uni.wroc.pl}
\address{Tomasz Rzepecki \orcid{https://orcid.org/0000-0001-9786-1648}}
\email{Tomasz.Rzepecki@math.uni.wroc.pl}


\begin{document}

	\begin{abstract}
		We obtain several fundamental results on finite index ideals and additive subgroups of rings as well as on model-theoretic connected components of rings, which concern generating in finitely many steps inside additive groups of rings.

		Let $R$ be any ring equipped with an arbitrary additional first order structure, and $A$ a set of parameters.
		We show that whenever $H$ is an $A$-definable, finite index subgroup of $(R,+)$, then $H+R\cdot H$ contains an $A$-definable, two-sided ideal of finite index. As a corollary, we obtain a positive answer to Question 4.9 of \cite{GJK22}: if $R$ is unital, then $\RRaT_A + \bar R \cdot \RRaT_A + \bar R \cdot \RRaT_A = \bar R^{00}_A$, which also implies that $\bar R^{00}_A=\bar R^{000}_A$, where $\bar R \succ R$ is a sufficiently saturated elementary extension of $R$, $\RRaT_A$ [resp. $\bar R^{00}_A$] denotes the smallest $A$-type-definable, bounded index additive subgroup [resp. ideal] of $\bar R$, and $\bar R^{000}_A$ is the smallest invariant over $A$, bounded index ideal of $\bar R$. If $R$ is of positive characteristic (not necessarily unital), we get a sharper result: $\RRaT_A + \bar R \cdot \RRaT_A = \bar R^{00}_A$. We obtain a similar result (but with more steps required) for finitely generated (not necessarily unital) rings. We obtain analogous results for topological rings. The above result for unital rings implies that the simplified descriptions of the definable (and so also classical) Bohr compactifications of triangular groups over unital rings obtained in Corollary 4.5 of \cite{GJK22} are valid for all unital rings.

		We also analyze many concrete examples, where we compute the number of steps needed to generate a group by $(\bar R \cup \{1\}) \cdot \RRaT_A$ and study related aspects,
		showing ``optimality'' of some of our main results and yielding answers to some natural questions.
	\end{abstract}

	\maketitle

	\section{Introduction}

	Our main results belong to algebraic model theory and concern fundamental questions on relationships between model-theoretic connected components of rings and their additive groups, yielding, for a large class of definable rings, a very simple description of the former in terms of the latter, which implies that the ``hierarchy'' of the ring components is flat in those rings (see Theorem~\ref{theorem: Main Theorem} below, as well as Theorem~\ref{theorem: main topological} for the topological variant).
	Simultaneously, this is very strongly related to some purely algebraic issues concerning finite index ideals and subgroups of arbitrary rings,
	and all our results have interesting, purely algebraic reformulations.
	In particular, we prove the following
	theorem, which is repeated later as Theorem~\ref{theorem: main theorem} (and see also Corollary~\ref{corollary: main corollary}).

	\begin{theorem}\label{theorem: main algebraic theorem}
		Let $R$ be an arbitrary ring $\emptyset$-definable in a structure $M$ and $A \subseteq M$. Then for every $A$-definable finite index subgroup $H$ of $(R,+)$, the set $H+R\cdot H$ contains an $A$-definable, two-sided ideal of $R$ of finite index.
	\end{theorem}
	Forgetting about definability and model theory, this specializes (by taking the full structure on $R$) to the following (to our knowledge, hitherto unknown) purely algebraic statement: For an arbitrary ring $R$, for every finite index subgroup $H$ of $(R,+)$, the set $H+R\cdot H$ contains a two-sided ideal of $R$ of finite index.

	Using the above theorem, we obtain the main result of this paper, which answers positively Question 4.9 of \cite{GJK22}. Before we state it, let us discuss the context. We are working in a first order structure $M$ whose monster (i.e.\ sufficiently saturated and homogeneous) model is denoted by $\bar M$; if $X$ is a set definable in $M$, then $\bar X$ always denotes its interpretation in $\bar M$. By $A$ we denote a small (i.e.\ of size smaller than the degree of saturation of $\bar M$) subset of $\bar M$.

	For a definable group $G$ the model-theoretic connected components ${\bar G}^0_A$, ${\bar G}^{00}_A$, and ${\bar G}^{000}_A$ (see Section~\ref{section: preliminaries} for the definitions) play a fundamental role in the analysis of definable groups as first order structures and in the so-called ``definable topological dynamics''. They also lead to certain natural definable compactifications of $G$; for example, $\bar G/{\bar G}^0_M$ is the universal definable, profinite
	compactification of $G$, whereas $\bar G/{\bar G}^{00}_M$ is the universal definable (so-called definable Bohr) compactification. Taking the full structure on $G$ (i.e.\ when all subsets of all finite
	Cartesian powers of $G$ are $\emptyset$-definable), one can remove the adjective ``definable'' obtaining model-theoretic descriptions of classical compactifications of $G$. In \cite{GJK22}, a basic theory of model-theoretic connected components of rings was developed; in particular, the quotients by these components yield various compactifications of the ring in question. It was used in the computation of definable (and classical) Bohr compactifications of some matrix groups. This development led to some fundamental questions on relationships between connected components of definable rings and their additive groups. We answer these questions in the present paper.

	Now, we state our main result, which will be deduced from Theorem~\ref{theorem: main algebraic theorem}
	(see Corollaries~\ref{corollary: 3 steps} and \ref{corollary: 0=00=000} for the full statement).

	\begin{theorem}\label{theorem: Main Theorem}
		Let $R$ be a $\emptyset$-definable ring.
		\begin{enumerate}
			\item If $R$ is unital, then $\RRaT_A + \RR \cdot \RRaT_A + \RR \cdot \RRaT_A = \RR^{00}_A$.
			\item If $R$ is of positive characteristic (not necessarily unital), then $\RRaT_A + \RR \cdot \RRaT_A = \RR^{00}_A$.
		\end{enumerate}
		Consequently, for both classes of rings, we have $\RRiD_A=\RRiT_A=\RRiI_A$.
	\end{theorem}

	In fact, we prove that, more generally, item (i) holds for all s-unital rings (see Definition~\ref{definition: s-unital}). Whenever we talk about unital rings in this introduction, we can weaken this assumption to s-unital rings, which we will not do just to avoid additional terminology.

	Let $J$ be the subgroup of $(\RR,+)$ generated by $(\RR \cup \{1\})\cdot \RRaT_A$.
	By Lemma 4.7 of \cite{GJK22}, we know that $J =\RRiI_A$. And by Lemma 4.8 of \cite{GJK22}, we know that the following conditions are equivalent:
	\begin{enumerate}[label=(\roman*)]
		\item $\RRiI_A$ is type-definable,
		\item $\RRiI_A$ is generated by $(\RR \cup \{1\})\cdot \RRaT_A$ in finitely many steps,
		\item $\RRiI_A = \RR^{00}_A$.
	\end{enumerate}
	(Strictly speaking, in \cite{GJK22}, $A$ is a model, but using \cite{Mas18}, one concludes that it works also over any set of parameters, see Fact~\ref{fact: occurrence of 000}.)
	Therefore, by Theorem~\ref{theorem: Main Theorem}, we get that the above equivalent conditions hold for all unital rings and also for all rings of positive characteristic. Hence, the answer to Question 4.6 of \cite{GJK22} is positive, and so the simplified formulas for the definable (so also for classical) Bohr compactifications of groups of upper unitriangular as well as invertible upper triangular matrices over unital rings given in \cite[Corollary 4.5]{GJK22} are valid for all unital rings.
	Theorem~\ref{theorem: Main Theorem} may also prove to be useful to compute Bohr compactifications of some other matrix groups over unital rings, or even more general groups which are in some way ``controlled'' by rings.

	One of the main questions which remain open is whether in Theorem~\ref{theorem: Main Theorem} one can drop the assumption that $R$ is unital or of positive characteristic. More precisely,

	\begin{question}\label{question: introduction, finite number of steps for all rings}
		Does $(\RR \cup \{1\})\cdot \RRaT_A$ generate a group in finitely many steps for an arbitrary ring $R$? If yes, is there a bound on the number of steps?
	\end{question}

	Theorem~\ref{theorem: Main Theorem} implies that for unital rings 3 steps, and for rings of positive characteristic 2 steps, suffice.
	We also prove (see Theorem~\ref{theorem: bound for fin gen. rings}) that whenever $R$ is generated by $n$ elements (where $n \in \omega$), then for any $A$ containing those elements, $n+2$ steps suffice.

	One can also ask whether 3 is an optimal bound on the number of steps for all unital rings. We give examples where exactly 2 steps are needed, e.g.\ $R\coloneqq \Z[X]$ equipped with the full structure (i.e.\ with predicates for all subsets of all finite Cartesian powers of $R$) is a unital ring with this property (see Example~\ref{example: Z[X]}); in Example~\ref{example:independent_functionals}, we construct a commutative, unital ring of characteristic 2 where exactly 2 steps are needed, so 2 is the optimal bound in the second item of Theorem {\ref{theorem: Main Theorem}. Thus, the remaining question is whether one can decrease the number of steps in the first item of Theorem~\ref{theorem: Main Theorem} from 3 to 2. (In Definition~\ref{definition:conditions}, we will introduce notions of \emph{half-integer} numbers of steps which lead to more precise information and are used throughout this paper, but in the introduction we omit this terminology just for simplicity.)

	In Lemma~\ref{lemma:Z[x]_no_iii1}, we find a finite index, additive subgroup $H$ of $R\coloneqq \Z[X]$ for which $R\cdot H$ does not contain a two-sided ideal of finite index. This shows that ``2 steps'' in Theorem~\ref{theorem: main algebraic theorem} is an optimal bound.

	We also analyze many other classical examples, computing the numbers of steps needed in the context of Theorem~\ref{theorem: main algebraic theorem}, in the sense of generating a group by $(\RR \cup \{1\})\cdot \RRaT_A$, and in other senses.

	We apply Theorem~\ref{theorem: Main Theorem} to get an analogous result for model-theoretic connected components of topological rings and their additive groups (for the relevant definitions see Section~\ref{section: topological rings},
	for the full statement see Theorem~\ref{theorem: topological case} and Corollary~\ref{corollary: 0=00=000 for top. rings}):

	\begin{theorem}\label{theorem: main topological}
		Let $R$ be a topological ring.
		\begin{enumerate}[topsep=0pt]
			\item If $R$ is unital, then $\RRaT_{\topo} + \RR \cdot \RRaT_{\topo} + \RR \cdot \RRaT_{\topo} = \RR^{00}_{\topo}$.
			\item If $R$ is of positive characteristic (not necessarily unital), then $\RRaT_{\topo} + \RR \cdot \RRaT_{\topo} = \RR^{00}_{\topo}$.
		\end{enumerate}
		Consequently, for both classes of rings, we have $\RRiD_{\topo}=\RRiT_{\topo}=\RRiI_{\topo}$.
	\end{theorem}

	This implies that $(\dagger \dagger)$ from Subsection 4.5 of \cite{GJK22} holds, and so one gets simplified formulas for the Bohr compactifications of topological groups $\UT_n(R)$ and $\T_n(R)$ over any unital, topological ring $R$.

	Coming back to the general situation when $R$ is any ring $\emptyset$-definable in $M$, a question arises whether for every type-definable, bounded index subgroup $H$ of $(\RR,+)$ the set $(\RR \cup \{1\}) \cdot H$ generates a group in finitely many steps. Theorems~\ref{theorem: Main Theorem} and \ref{theorem: main topological} tell us that for unital rings and for rings of positive characteristic it is true for $H\coloneqq \RRaT_A$ and $H\coloneqq \RRaT_{\topo}$ when $R$ is topological. In general, we give a negative answer
	(even for commutative, unital rings of positive characteristic)
	by finding an appropriate $H$ for $R \coloneqq \Z_2^\omega$ equipped with the full structure (see Corollary~\ref{corollary: infinitely many steps in in the power of Z2}). Using the same ring, we also show that it is not the case that there exists $n \in \omega$ such that for every ($\emptyset$-definable) finite index subgroup $H$ of $(R,+)$, the set $(R \cup \{1\}) \cdot H$ generates a group in $n$-steps (see Example~\ref{example: Z2omega}); in particular, in Theorem~\ref{theorem: main algebraic theorem}, one cannot strengthen the conclusion to saying that $H +R\cdot H$ is a two-sided (or just left) ideal.

	We said that Theorem~\ref{theorem: Main Theorem} will be deduced from our algebraic Theorem~\ref{theorem: main algebraic theorem}. But in Section~\ref{section: conditions}, we will see that the connection is much stronger, e.g.\ for rings of positive characteristic both theorems are easily seen to be equivalent. In Section~\ref{section: conditions}, we will distinguish more conditions closely related to both theorems, in particular
	(in Proposition~\ref{proposition: equivalence of finite numbers of steps}(2))
	we will show that Theorem~\ref{theorem: main algebraic theorem} is equivalent to the following version of Theorem~\ref{theorem: Main Theorem}.

	\begin{theorem}\label{theorem: Main Theorem for (R,+)0}
		Let $R$ be a $\emptyset$-definable ring. Then $(\RR,+)^{0}_A + \RR \cdot (\RR,+)^{0}_A = \RR^{0}_A$, where $(\RR,+)^{0}_A$ [resp. $\RR^0_A$] is the intersection of all $A$-definable, finite index additive subgroups [resp. ideals] of $\RR$.
	\end{theorem}

	Theorem~\ref{theorem: main algebraic theorem} can be viewed as a purely algebraic restatement of Theorem~\ref{theorem: Main Theorem for (R,+)0}. In Section~\ref{section: final comments}, we will describe an algebraic restatement of Theorem~\ref{theorem: Main Theorem}, Corollary~\ref{corollary: 3 steps}, and Theorem~\ref{theorem: bound for fin gen. rings}, and, more generally, of a potential positive answer to Question~\ref{question: introduction, finite number of steps for all rings}.
	Roughly speaking, this restatement says that for an appropriate $n$, for any Bohr neighborhood $D$ in the group $(R,+)$ there exists a
	weak ``ring version of a Bourgain system'' in the $n$-fold sum $(R \cup\{1\}) \cdot D + \dots + (R \cup\{1\}) \cdot D$.

	An important consequence of Pontryagin duality is the fact that each unital, compact, Hausdorff topological ring is profinite. From this, one gets that for $R$ unital, $\RR^{00}_A = \RR^0_A$. The same is true for rings of positive characteristic (then also $\RRaT_A = (\RR,+)^{0}_A$) which is pointed out in Corollary~\ref{corollary: 0=00 in rings}. Note that in general it may happen that $\RR^{00}_A \ne \RR^0_A$, e.g.\ for $R$ being the ring of integers but with the zero multiplication, equipped with the full structure, we have $\RR^{00}_R =\RRaT_R \ne (\RR,+)^{0}_R = \RR^0_R$ (where the middle inequality follows from \cite[Corollary 2.7]{GJK22}).

	\subsection*{Structure of the paper}
	In the preliminaries, we recall the necessary definitions and useful facts, extending the context of some of them. In Section~\ref{section: conditions}, we distinguish several conditions concerning generating in a given, finite number of steps (similar to those in the conclusions of Theorems~\ref{theorem: main algebraic theorem} and~\ref{theorem: Main Theorem}) and describe the relationships between them which are used in further sections to obtain our main results and to study examples.
	In Section~\ref{section: main results}, we prove the main results: Theorems~\ref{theorem: main algebraic theorem} and~\ref{theorem: Main Theorem}. In Section~\ref{section: comm. fin. gen. rings}, we give simpler proofs of these results, but for commutative, finitely generated, unital rings, and with a weaker conclusion (more steps to generate a group). In Section~\ref{section: fin. gen. rings}, we extend Theorem~\ref{theorem: Main Theorem} to arbitrary finitely generated (not necessarily unital) rings, but again with more steps required to generate a group. In Section~\ref{section: topological rings}, we discuss model-theoretic connected components of topological rings and prove Theorem~\ref{theorem: main topological}. In Section~\ref{section: examples}, we analyze many examples, showing ``optimality'' of Theorems~\ref{theorem: main algebraic theorem} and~\ref{theorem: Main Theorem}(2), and yielding negative answers to
	some natural questions discussed in this introduction.
	In Section \ref{section: final comments}, we discuss purely algebraic restatements of our main results.

	\section{Preliminaries}\label{section: preliminaries}

	We are working in a model $M$ of a complete, first order theory $T$ whose monster model is denoted by $\bar M$. Recall that by a monster model we mean a $\kappa$-saturated and strongly $\kappa$-homogeneous model for a sufficiently large
	strong limit cardinal $\kappa$ (for our purposes, any one greater than $|T|$ will suffice).
	When we are talking about definable groups or rings, we mean that they are definable in $M$ or in $\bar M$;
	the reader may assume that $M$ is simply the group or ring in question (possibly with some additional structure) without much loss of generality.
	For a set $D$ definable in $M$, $\bar D$ will stand for its interpretation in $\bar M$. We will say that $D$ is equipped with the \emph{full structure} if all subsets of all finite Cartesian powers of $D$ are $\emptyset$-definable.
	Usually, $A$ will denote a fixed subset of $M$.\footnote{However, most of the results can be easily adapted to treat the case of an arbitrary small $A\subseteq \bar M$.}
	The term ``bounded'' will mean ``less than $\kappa$''.

	Let us recall all the relevant notions of connected components. Let $R$ be a $\emptyset$-definable group [resp. ring].

	\begin{itemize}
		\item ${\bar R}^0_A$ is the intersection of all $A$-definable, finite index subgroups [ideals] of $\bar R$.
		\item ${\bar R}^{00}_A$ is the smallest $A$-type-definable, bounded index subgroup [ideal] of $\bar R$.
		\item ${\bar R}^{000}_A$ is the smallest invariant over $A$, bounded index subgroup [ideal] of $\bar R$.
	\end{itemize}

	Note that we did not specify whether the ideals above are left, right, or two-sided. This is because of Proposition 3.6, Corollary 3.7, and Proposition 3.10 from \cite{GJK22} which tell us that

	\begin{fact}\label{fact: ideal component}
		The above components of the ring $\bar R$ do not depend on the choice of the version (left, right, or two-sided) of the ideals. Moreover, instead of ``ideal'' we can equivalently write ``subring'' in the above definitions.\footnote{It is easy to see (cf.\ for example \cite[Lemma 2.2(3)]{gis}) that the analogue for groups is also true, i.e.\ $\bar R^0_A, \bar R^{00}_A, \bar R^{000}_A$ are always normal subgroups of $\RR$, even if $R$ is a non-abelian group.}
	\end{fact}

	Of course, ${\bar R}^{000}_A \leq {\bar R}^{00}_A \leq {\bar R}^0_A$. It is easy to see (cf. \cite[Lemma 2.2(1)]{gis}) that the indices of all these components in $\bar R$ are in fact bounded by $2^{|T|+|A|}$. If $R$ is a definable group [or ring] and $S$ is a type-definable, normal subgroup [ideal] in $\RR$ of bounded index, then $\RR/S$ is equipped with the \emph{logic topology}: closed sets are those whose preimages under the quotient map are type-definable. This makes the quotient $\RR/S$ a compact, Hausdorff topological group [ring]. Then $\RR/\RR^0_R$ is the universal definable, profinite compactification, and $\RR/\RR^{00}_R$ the universal definable compactification (so-called definable Bohr compactification) of the group [ring] $R$. (See \cite{GPP14} and \cite{GJK22} for definitions and more details; a nice, short discussion on definable Bohr compactifications of groups can be found in the preliminaries in \cite{GJK22}.) So for example the equality $\RRiT_R=\RRiD_R$ means precisely that both compactifications coincide. When $R$ is equipped with the full structure, we can
	erase
	the adjective ``definable'' and we get classical notions of compactification (described in a model-theoretic way).

	Since in this paper $R$ will denote a definable ring, we will write $\RRiD_A$, etc. to denote the components of $\RR$ as a ring, and $(\RR,+)^0_A$, etc. to denote the components of the additive group $(\RR,+)$.

	Recall that a definable group $G$ is said to be \emph{definably amenable} if there is a left invariant, finitely additive probability measure on the algebra of definable subsets of $G$. Each abelian group is amenable as a discrete group so also definably amenable.

	Recall the notion of a thick set from \cite[Definition 3.1]{gis}.
	For a positive integer $n$, a subset $D$ of a group is said to be \emph{$n$-thick} if it is symmetric and for any elements $g_0,\dots,g_{n-1}$ there exist $i<j<n$ with $g_i^{-1}g_j \in D$; it is called \emph{thick}, if it is $n$-thick} for some $n$. A subset $D$ of $G$ is called (left) \emph{generic} if finitely many (left) translates of it cover $G$. Each thick subset is clearly generic.

	\begin{fact}\label{fact: G00=G000}
		If $G$ is a definably amenable group (e.g.\ an
		abelian	group), then $\bar G^{00}_A=\bar G^{000}_A$.
	\end{fact}

	\begin{proof}
		For $A=M$ (where $G=G(M)$) this is \cite[Theorem 3.3]{KrPi}. For an arbitrary $A\subseteq M$ one can argue as follows. By Lemmas 2.2(2) and 3.3 of \cite{gis}, ${\bar G}^{000}_A$ is generated by the intersection of all $A$-definable, thick subsets of $\bar G$.
		So, by compactness, it is enough to show that for every $A$-definable, thick subset $D$ of $G$, ${\bar G}^{00}_A \subseteq \bar D^{8}$.
		Claim 1 in the proof of \cite[Lemma 3.5]{KrPi} (which is based on the main result of \cite{MaWa}) yields a descending sequence $D^4=C_1 \supseteq C_2 \supseteq \dots$ of generic, symmetric, $M$-definable subsets of $G$ such that $C_{i+1}^2 \subseteq C_i$ for all $i$. Then $\bigcap_i \bar C_i$ is an $M$-type-definable subgroup of $\bar G$ of bounded index and contained in $\bar D^4$. Therefore, ${\bar G}^{00}_M \subseteq \bar D^4$. Since $D$ is $A$-definable, by \cite[Theorem 5.2]{Mas18}, we conclude that ${\bar G}^{00}_A \subseteq \bar D^8$.
	\end{proof}

	Let $R$ be a definable ring.
	Using Fact~\ref{fact: G00=G000} together with a result of Newelski, the following two observations were made in \cite{GJK22} (see Lemmas 4.7 and 4.8 in there), whose proofs we recall for the reader's convenience.

	Put $X_\leftarrow \coloneqq (\bar R \cup \{1\}) \cdot \RRaT_A$, $X_\rightarrow\coloneqq \RRaT_A \cdot (\bar R \cup \{1\})$, $X_\leftrightarrow\coloneqq (\bar R \cup \{1\}) \cdot \RRaT_A \cdot (\bar R \cup \{1\})$.

	\begin{fact}\label{fact: occurrence of 000}
		The subgroup $J$ of $(\bar R, +)$ generated by
		$X_i$ is precisely $\RRiI_A$ for any $i \in \{\leftarrow,\rightarrow,\leftrightarrow\}$.
	\end{fact}
	\begin{proof}
		As $\RRa$ is abelian, $\RRaT_A = \RRaI_A \subseteq \RRiI_A$ by Fact~\ref{fact: G00=G000}. Hence, $J$ is contained in $\RRiI_A$ by Fact~\ref{fact: ideal component}. On the other hand, $J$ is an $A$-invariant left [or right or two-sided] ideal which contains $\RRaT_A$ and so has bounded index, hence it must contain $\RRiI_A$ by Fact~\ref{fact: ideal component}.
	\end{proof}

	\begin{fact}\label{fact: conditions equivalent to 000=00}
		For any $i \in \{\leftarrow,\rightarrow,\leftrightarrow\}$, the following conditions are equivalent.
		\begin{enumerate}[label=(\roman*)]
			\item $\RRiI_A$ is type-definable.
			\item $\RRiI_A$ is generated by $X_i$ in finitely many steps.
			\item $\RRiI_A = \RRiT_A$.
		\end{enumerate}
	\end{fact}
	\begin{proof}
		The implication (i) $\rightarrow$ (ii) follows from Theorem 3.1 of \cite{Ne03} (and Fact~\ref{fact: occurrence of 000}); and (ii) $\rightarrow$ (i) is trivial. The equivalence (i) $\leftrightarrow$ (iii) is trivial by the definitions of $\RRiI_A$ and $\RRiT_A$.
	\end{proof}

	As a corollary we see that the statements ``$\RRiI_A$ is generated by $X_i$ in finitely many steps'' are equivalent for $i \in \{\leftarrow,\rightarrow,\leftrightarrow\}$.

	Question 4.9 of \cite{GJK22} asks whether these equivalent conditions hold for all unital rings.
	As discussed in the introduction, we will answer this question in the affirmative, showing that 3 (and for positive characteristic even 2) steps are enough to generate $\RRiI_A$.

	The next fact was observed in \cite[Proposition 3.6]{GJK22}.
	\begin{fact}\label{fact: 0=00 in rings}
		If $R$ is unital, then $\RRiT_A =\RRiD_A$.
	\end{fact}

	The above fact is based on \cite[Proposition 5.1.2]{RZ} which says that each compact, Hausdorff, unital ring is profinite, the proof of which goes through when instead of the existence of 1 one assumes that for every $c \in R$ there exists $r_c \in R$ such that $r_c c=c$.

	However, we can say more.

	\begin{definition}\label{definition: s-unital}
		Let $R$ be a ring. We say that it is \emph{left s-unital} if for every $r\in R$ we have $r\in Rr$, \emph{right s-unital} if for every $r\in R$ we have $r\in rR$, and \emph{s-unital} if it is both left and right s-unital.
	\end{definition}

	\begin{example}
		Each unital ring is s-unital. The converse is false: ${\mathbb Z}_2^{\oplus \omega}$ is s-unital but not unital. Furthermore, the ring $C_c(\R)$ of compactly supported, continuous functions on the real line is s-unital, but it has no nonzero idempotents.
	\end{example}

	\begin{definition}
		If $R$ is a ring and $r\in R$ is nonzero, then we say that $r$ is a \emph{total left zero divisor} if $rR=0$, and we say that $r$ is a \emph{total right zero divisor} if $Rr=0$.
	\end{definition}

	Note that in a left s-unital ring, there are no total right zero divisors, and in a right s-unital ring, there are no total left zero divisors.

	\begin{fact}
		\label{fact:anzai}
		If $R$ is a compact, Hausdorff ring which has no total left zero divisors or no total right zero divisors, then $R$ is a profinite ring.
	\end{fact}
	\begin{proof}
		This is \cite[Theorem 3]{Anz43}.
	\end{proof}

	The next corollary is a generalization of Fact~\ref{fact: 0=00 in rings}.
	\begin{corollary}\label{corollary: 0=00 in rings}
		If $R$ is either of positive characteristic or unital or just left [or right] s-unital, then $\RRiT_A =\RRiD_A$.
	\end{corollary}
	\begin{proof}
		As mentioned before, $\RRiT_A=\RRiD_A$ if and only if $\RR/\RRiT_A$ is profinite (e.g.\ see the proof of \cite[Proposition 3.6]{GJK22}).
		If $R$ is left s-unital or right s-unital, then clearly so is $\RR$ and its every quotient, including $\RR/\RRiT_A$. Since a left [right] s-unital ring cannot have right [left] total zero divisors, it follows by Fact~\ref{fact:anzai} that $\RR/\RRiT_A$ is profinite, so we are done.

		If $R$ is of characteristic $m>0$, then expand $M$ by the additional sort for the finite ring $\Z/m\Z$ with named elements and its natural action on $R$. Then the induced structure on $M$ is interdefinable with the original one (without parameters). Consider the unitization $S\coloneqq R\oplus \Z/m\Z$ of $R$; this is a $\emptyset$-definable ring. It is clear that ${\bar S}^{00}_A = \RRiT_A \times \{0\}$ and ${\bar S}^{0}_A = \RRiD_A \times \{0\}$. So it is enough to show that ${\bar S}^{00}_A = {\bar S}^{0}_A$, but this follows from the first part, as $S$ is unital.
	\end{proof}

	As mentioned before, by \cite{GPP14}, we know that if $G$ is a (discrete) group considered with the full structure, then the (classical) Bohr compactification of $G$ coincides with $\G/\G^{00}_G=\G/\G^{00}_\emptyset$.
	If $G$ lives in a model $M$ equipped with the full structure, then this Bohr compactification is also $\G/\G^{00}_M$.

	The next well-known result is based on Pontryagin duality. For a proof see e.g.\ \cite[Fact 2.6]{GJK22}.

	\begin{fact}
		\label{fact:bohrIfOnly}
		Assume $G$ is a discrete abelian group. Then its Bohr compactification is
		profinite if and only if $G$ has finite exponent.
	\end{fact}

	\begin{corollary}\label{corollary: finite exponent implies 00=0}
		If $G$ is a group of finite exponent defined in an arbitrary structure, then $\bar G/{\bar G}^{00}_A$ is profinite; equivalently, ${\bar G}^{00}_A = {\bar G}^{0}_A$.
	\end{corollary}
	\begin{proof}
		This is \cite[Corollary 2.8]{GJK22}.\footnote{As pointed out by a referee, the abelianity hypothesis included in \cite{GJK22} is not necessary, since a compact Hausdorff group of finite exponent is always profinite.}
	\end{proof}

	Let us observe that the first part of Corollary~\ref{corollary: 0=00 in rings} follows from Corollary~\ref{corollary: finite exponent implies 00=0}. Indeed, the assumption that $R$ has positive characteristic means that $(R,+)$ is of finite exponent, so $\RRaT_A = (\RR,+)^{0}_A$. But this always implies that $\RRiT_A = \RRiD_A$. Namely, $\RR/\RRaT_A = \RR/(\RR,+)^{0}_A$ is a profinite group which maps onto $\RR/\RRiT_A$ which thus must be also profinite as a group. This implies that $\RR/\RRiT_A$ is a profinite ring, so $\RRiT_A =\RRiD_A$.

	Throughout the paper, when $X$ is a subset of an
	abelian group, $X^{+n}$ will denote the $n$-fold sum $X + \dots +X$.

	\section{Generating in finitely many steps --- equivalent conditions}\label{section: conditions}

	In this section, we formulate several conditions concerning generating in finitely many steps and study relationships between them, showing that they are equivalent in some general situations. In Section~\ref{section: main results}, we will show that condition \nref{condition:group_R00}{1\frac12} defined below holds for all rings
	which, together with the results of this section, will allow us to automatically deduce the other conditions, with the appropriate indices and for the appropriate rings, and in particular Theorem~\ref{theorem: Main Theorem}. The relationships established in this section are also used in the analysis of some concrete examples in Section~\ref{section: examples}.

	As usual, $R$ is a $\emptyset$-definable ring. For a subset $D$ of $R$, let:
	\[
	D_\leftarrow\coloneqq (R\cup \{1\})\cdot D, \; D_\rightarrow\coloneqq D\cdot (R \cup \{1\}), \textrm{ and } D_\leftrightarrow\coloneqq (R \cup \{1\})\cdot D\cdot (R\cup \{1\}).
	\]
	Analogous notation will be used for $D \subseteq \bar R$ (note that there is some ambiguity here, since $R\subseteq \bar R$, but the meaning should be clear from the context in each case). In particular, for $X\coloneqq \RRaT_A$, we get $X_\leftarrow$, $X_\rightarrow$, and $X_\leftrightarrow$ defined before Fact~\ref{fact: occurrence of 000}. Let also:
	\[
	X_\leftarrow' \coloneqq (\bar R \cup \{1\}) \cdot (\bar R,+)^{0}_A, \; X_\rightarrow'\coloneqq (\bar R,+)^{0}_A \cdot (\bar R \cup \{1\}), \textrm{ and }X_\leftrightarrow'\coloneqq (\bar R \cup \{1\}) \cdot (\bar R,+)^{0}_A \cdot (\bar R \cup \{1\}).
	\]
	Remembering that $J$ was defined as the subgroup of $(\bar R,+)$ generated by $X_i$
	(which were defined analogously to $X_i'$, but with $\RRaT_A$ in place of $\RRaD_A$),
	we define $J'_i$ as the subgroup generated by $X_i'$, for $i \in \{\leftarrow,\rightarrow,\leftrightarrow\}$.
	Then $J'_i$ is an $A$-invariant, left [or right or two-sided] ideal of $\RR$ of bounded index which is contained in $\RR^0_A$. While by Fact \ref{fact: occurrence of 000} we know that $J=\RR^{000}_A$ regardless of the choice of $i$ (and that is why there is no need to write $J_i$), at this moment we do not have an analogous fact about $J'_i$. However, from Corollary \ref{corollary: which conditions hold for rings}(1), it will follow that $J_i' = \RR^{0}_A$, regardless of the choice of $i$.

	By Fact~\ref{fact: ideal component}, in the next definition we can equivalently talk about left, right, and two-sided ideals, so will be skipping the adjectives before ideals.

	\begin{definition}
		\label{definition:conditions}
		We introduce the following conditions for any natural number $n>0$ and $i \in \{\leftarrow,\rightarrow,\leftrightarrow\}$.
		\begin{enumerate}[label=\textrm{(\roman{*})}$_n$, ref = \textrm{(\roman{*})}]
			\item
			\label{condition:X}
			$J(=\RRiI_A)=X_i^{+n}$ (equivalently, $X_i^{+n} = \RRiT_A$).
			\item
			\label{condition:def_R00}
			For every (thick) $A$-definable subset $D$ of $R$ such that $\bar D \supseteq \RRaT_A$, $D_i^{+n}$ contains an $A$-definable, finite-index ideal.
			\item
			\label{condition:group_R00}
			For every $A$-definable, finite index subgroup $H$ of $(R,+)$, $H_i^{+n}$ contains an $A$-definable ideal of finite index.
			\item
			\label{condition:X'_subgroup}
			$J'_i=(X_i')^{+n}$.
			\item
			\label{condition:X'_R0}
			$(X_i')^{+n}=\bar R^{0}_A$.
		\end{enumerate}
		We also consider conditions with half-integer subscripts. E.g.\ \nref{condition:X}{k\frac12}: $J= X+X_i^{+k}$, where $X\coloneqq \RRaT_A$; \nref{condition:def_R00}{k\frac12}: for every (thick) $A$-definable $D$ with $\bar D\supseteq \RRaT_A$, the set $D+D_i^{+k}$ contains an $A$-definable ideal of finite index. The others are defined analogously.

		In order to unify the notation, for $n=k\frac12$ (where $k \in \N$), $D_i^{+n}$ will stand for $D + D_i^{+k}$, etc.
	\end{definition}

	The equivalence of the two formulations of \nref{condition:X}{n} follows immediately from Facts~\ref{fact: occurrence of 000} and \ref{fact: conditions equivalent to 000=00}, but it is much easier: $J$ is clearly an $A$-invariant, bounded index, left [or right or two-sided] ideal of $\RR$ contained in $\RRiT_A$ and containing $X_i^{+n}$, so $J=X_i^{+n}$ if and only if $X_i^{+n} = \RRiT_A$.

	One should keep in mind that since each $\RRaT_A$, $\RRaD_A$, $\RRiT_A$, $\RRiD_A$ is an intersection of $A$-definable symmetric sets, and each $A$-definable symmetric set containing any of them is thick, each of these sets is the intersection of $A$-definable thick sets containing it.

	\begin{proposition}\label{proposition: equality of components and fin. many steps}
		If $n$ is a positive half-integer, then:
		\begin{enumerate}
			\item \nref{condition:def_R00}{n} holds if and only if both \nref{condition:X}{n} holds and $\RRiT_A=\RRiD_A$,
			\item $\RRiT_A=\RRiI_A$ if and only if for some $n$ we have \nref{condition:X}{n},
			\item $\RRiD_A=\RRiI_A$ if and only if for some $n$ we have \nref{condition:def_R00}{n}.
		\end{enumerate}
	\end{proposition}

	\begin{proof}
		(1) ($\rightarrow$). By \nref{condition:def_R00}{n}, we get $\RRiD_A \subseteq X_i^{+n}$, and trivially $X_i^{+n}\subseteq J$, which is $\RRiI_A$ by Fact~\ref{fact: occurrence of 000}.
		On the other hand, $\RRiI_A \subseteq \RRiT_A \subseteq \RRiD_A$. So we get $\RRiI_A= X_i^{+n} = \RRiT_A = \RRiD_A$.\\
		($\leftarrow$). By \nref{condition:X}{n} and $\RRiT_A=\RRiD_A$, for any $D$ from \nref{condition:def_R00}{n}, $\RRiD_A \subseteq \bar D_i^{+n}$. So the conclusion follows by compactness.

		(2) is a part of Fact~\ref{fact: conditions equivalent to 000=00}.

		(3) follows from (1) and (2).
	\end{proof}

	\begin{lemma}
		\label{lemma:s-unital_thick_kR}
		If $R$ is left [right] s-unital, then for every thick set $D\subseteq R$ there is $N>0$ such that $NR\subseteq R\cdot D$ [resp. $NR\subseteq D\cdot R$]. If $R$ is s-unital, then we get $NR\subseteq R\cdot D \cap D\cdot R \subseteq R\cdot D\cdot R$.
	\end{lemma}
	\begin{proof}
		Let us prove the left version; the right version is similar, and the two-sided version follows from the one-sided versions (note that s-unitality implies that $R\cdot D \cup D\cdot R \subseteq R\cdot D\cdot R$).

		Fix an $n$-thick set $D$. We will show that $n!R\subseteq R\cdot D$. Fix an arbitrary $r\in R$. Choose $s \in R$ with $r=sr$ (by left s-unitality). By $n$-thickness of $D$, there is some $k \in \{1,\dots,n-1\}$ such that $kr\in D$. Then $k$ divides $n!$, and it follows that \[n!r = n!sr = \left(\frac{n!}{k}\right)skr\in R\cdot D.\qedhere\]
	\end{proof}

	\begin{proposition}\label{proposition: equivalence of finite numbers of steps}
		If $n$ is a positive half-integer, then:
		\begin{enumerate}
			\item \nref{condition:def_R00}{n} implies \nref{condition:X}{n} and \nref{condition:group_R00}{n},
			\item conditions \nref{condition:group_R00}{n}, \nref{condition:X'_subgroup}{n}, and \nref{condition:X'_R0}{n} are all equivalent,
			\item if $R$ is of positive characteristic, then all conditions \nref{condition:X}{n} -- \nref{condition:X'_R0}{n} are equivalent,
			\item if $R$ is s-unital, then \nref{condition:def_R00}{n} is equivalent to \nref{condition:X}{n}, and \nref{condition:group_R00}{n} implies \nref{condition:def_R00}{n+1}.
		\end{enumerate}
	\end{proposition}

	\begin{proof}
		(1) The implication \nref{condition:def_R00}{n} $\rightarrow$ \nref{condition:X}{n} follows from Proposition~\ref{proposition: equality of components and fin. many steps}(1), and \nref{condition:def_R00}{n} $\rightarrow$ \nref{condition:group_R00}{n} is trivial.

		(2) The implication \nref{condition:X'_R0}{n} $\rightarrow$ \nref{condition:X'_subgroup}{n} is trivial.\\
		\nref{condition:group_R00}{n} $\rightarrow$ \nref{condition:X'_R0}{n}. By \nref{condition:group_R00}{n} and compactness, we get that ${\bar R}^{0}_A\subseteq (X_i')^{+n}$, so \nref{condition:X'_R0}{n} follows from the obvious observation that $(X_i')^{+n} \subseteq J'_i \subseteq \bar R^{0}_A$.\\
		\nref{condition:X'_subgroup}{n} $\rightarrow$ \nref{condition:group_R00}{n}. By \nref{condition:X'_subgroup}{n}, $J'_i$ is an $A$-type-definable, left [or right or two-sided] ideal containing $(\bar R,+)^0_A$.
		So $J'_i$ is an intersection $\bigcap_j G_j$ of some $A$-definable subgroups of $(\RR,+)$ of finite index. Note that $\Stab_{\bar R}(G_j)\coloneqq \{ r \in \bar R: rG_j \subseteq G_j\}$ for $i\in \{\rightarrow,\leftrightarrow\}$ and $\Stab_{\bar R}'(G_j)\coloneqq \{ r \in \bar R: G_j r\subseteq G_j\}$ for
		$i={\leftarrow}$ has finite index, because it is definable and contains $J'_i$ which is of bounded index. Hence, by \cite[Lemma 3.9 and Proposition 3.10]{GJK22}, we get that each group $G_j$ contains an $A$-definable, finite index, two-sided ideal. So, $\bar R^{0}_A \subseteq J'_i$. Therefore, $J'_i =\bar R^{0}_A$.
		Since $J'_i= (X_i')^{+n}$, we get \nref{condition:group_R00}{n} by
		compactness.

		(3) By Corollary~\ref{corollary: 0=00 in rings}, ${\bar R}^{00}_A={\bar R}^0_A$, so the implication \nref{condition:X}{n} $\rightarrow$ \nref{condition:def_R00}{n} follows from Proposition~\ref{proposition: equality of components and fin. many steps}(1).
		To show \nref{condition:group_R00}{n} $\rightarrow$ \nref{condition:def_R00}{n}, consider any $D$ from \nref{condition:def_R00}{n}. By Corollary~\ref{corollary: finite exponent implies 00=0}, $({\bar R},+)^{00}_A=({\bar R},+)^0_A$, so we can find a finite index, $A$-definable subgroup $H\subseteq D$. By \nref{condition:group_R00}{n}, $H_i^{+n}$ contains an $A$-definable, finite-index ideal, and so does $D_i^{+n}$.

		(4) Again by Corollary~\ref{corollary: 0=00 in rings}, ${\bar R}^{00}_A={\bar R}^0_A$, so the implication \nref{condition:X}{n} $\rightarrow$ \nref{condition:def_R00}{n} follows from Proposition~\ref{proposition: equality of components and fin. many steps}(1).
		To show \nref{condition:group_R00}{n} $\rightarrow$ \nref{condition:def_R00}{n+1}, consider any $D$ from \nref{condition:def_R00}{n}. Since $D$ is thick, Lemma~\ref{lemma:s-unital_thick_kR} yields a positive natural number $N$ such that the two-sided ideal $NR$ is contained in $R\cdot D \cap D\cdot R \subseteq R\cdot D\cdot R$.
		Then $S\coloneqq R/NR$ is an interpretable ring of characteristic $N>0$, and so $({\bar S},+)^{00}_A = ({\bar S},+)^0_A$ by Corollary~\ref{corollary: finite exponent implies 00=0}. On the other hand, $({\bar S},+)^{00}_A = ({\bar R},+)^{00}_A/N{\bar R} \subseteq {\bar D}/N{\bar R}$. Thus, there exists an $A$-definable, finite index subgroup $H$ of $(R,+)$ such that $H \subseteq D + NR$.
		By \nref{condition:group_R00}{n}, $H_i^{+n}$ contains an $A$-definable ideal of finite index. This is enough, as $H_i^{+n}$ is contained in $D_i^{+n} +NR \subseteq D_i^{+n}+ (R\cdot D \cap D\cdot R \cap R\cdot D\cdot R) \subseteq D_i^{+(n+1)}$.
	\end{proof}

	\begin{remark}
		\label{remark:equivalence_one-sided}
		We also have an analogue of Proposition~\ref{proposition: equivalence of finite numbers of steps}(4) when $R$ is only left [or right] s-unital. For instance, if $R$ is left s-unital, then \nref{condition:X}{n} $\rightarrow$ \nref{condition:def_R00}{n} for all $i\in \{\leftarrow, \rightarrow, \leftrightarrow\}$, and \nref{condition:group_R00}{n} $\rightarrow$ \nref{condition:def_R00}{n+1} for $i\in \{\leftarrow, \leftrightarrow\}$, both with essentially the same proof.
	\end{remark}

	In the next section, we will see that \nref{condition:group_R00}{1\frac12} always holds. However, any ring $R$ with zero multiplication and $\RRaT_A \ne \RRaD_A$ is an example where \nref{condition:X}{\frac12} holds, but \nref{condition:def_R00}{n} fails for all $n$ (by Proposition~\ref{proposition: equality of components and fin. many steps}(1), as $\RRiT_A=\RRaT_A \ne \RRaD_A=\RRiD_A$). Particular examples of this form are $\Z$ equipped with the full structure or the circle $S^1$ treated as a group
	definable in $(\R,+,\cdot)$. In Example~\ref{example: XZ[X]}, we will check that for $R\coloneqq X\Z[X]$ (a ring without zero divisors) equipped with the full structure \nref{condition:def_R00}{n} also fails for all $n$.

	\begin{question}
		Let $R$ be arbitrary. Does \nref{condition:group_R00}{n}
		imply that \nref{condition:X}{m} holds for some $m$? Is it true with $m=n+1$ or even with $m=n$?
	\end{question}

	As remarked after Fact~\ref{fact: conditions equivalent to 000=00}, condition $\exists n$\nref{condition:X}{n} does not depend on the choice of $i \in \{\leftarrow,\rightarrow,\leftrightarrow\}$, because it is equivalent to $\bar R^{00}_A=\bar R^{000}_A$.
	Hence, by Proposition~\ref{proposition: equivalence of finite numbers of steps}, the same is true for all other conditions from Definition~\ref{definition:conditions} when $R$ is $s$-unital or of positive characteristic.

	\section{Main results}\label{section: main results}
	We will prove here Theorems~\ref{theorem: main algebraic theorem} and \ref{theorem: Main Theorem} (see Theorem~\ref{theorem: main theorem}, and Corollaries~\ref{corollary: main corollary}, \ref{corollary: 3 steps}, and \ref{corollary: 0=00=000}).

	\begin{definition}
		\label{def:coset_indep}
		We will say that two subgroups $H_1$ and $H_2$ of an abelian group $G$ are \emph{coset-independent} if any coset of $H_1$ intersects any coset of $H_2$. They are \emph{coset-dependent} if they are not coset-independent.
	\end{definition}

	\begin{remark}\label{remark: coset independent}
		Let $G$ be an abelian group and $H_1,H_2 \leq G$. The following conditions are equivalent.
		\begin{enumerate}[label=(\roman*)]
			\item $H_1$ and $H_2$ are coset-independent.
			\item $H_1$ intersects any coset of $H_2$.
			\item $H_1+H_2 =G$.
		\end{enumerate}
		Thus, $H_1$ and $H_2$ are coset-dependent if and only if $H_1+H_2$ is a proper subgroup of $G$.
	\end{remark}
	\begin{proof}
		Easy exercise.
	\end{proof}

	\begin{lemma}\label{lemma: generation in finitely many steps by generic}
		If $D$ is a generic subset of a group $G$, then $\langle D \rangle$ is generated in finitely many steps.
	\end{lemma}

	\begin{proof}
		Equip $G$ with the full structure, and let $\bar G \succ G$ be a monster model and $\bar D$ the interpretation of $D$ in $\bar G$. Then the group $\langle \bar D \rangle$ is $\bigvee$-definable and of finite index, hence, by compactness, it is definable and generated in finitely many steps.

		Alternatively, this can be proved constructively as follows (we thank an anonymous referee for suggesting this alternative proof): let $n$ be such that $G$ is covered by $n$ translates of $D$, and put $E=D\cup \{1\}\cup D^{-1}$. We will show that $E^{3n}=\langle D\rangle$: if not, we could choose for each $1\leq i\leq 3n+1$ an element $g_i\in E^i\setminus E^{i-1}$. By pigeonhole, two of $g_1,g_4,\dots, g_{3n+1}$, say $g_i$ and $g_j$ (where $i<j$), are in the same $gD$ for some $g\in G$. Then $i+2\leq j-1$ and $g_i^{-1}g_j\in D^{-1}g^{-1}gD=D^{-1}D\subseteq E^2$. Thus $g_j\in g_iE^2\subseteq E^iE^2=E^{i+2}\subseteq E^{j-1}$, a contradiction.
	\end{proof}

	\begin{lemma}\label{lemma: main lemma}
		Let $R$ be a ring and $H$ be a finite index subgroup of $(R,+)$. Let $T$ be the collection of all $a \in R/H$ such that for every finite index subgroup $K$ of $(R,+)$ we have $a \in RK/H$. Then $T$ is a subgroup of $(R/H,+)$.
	\end{lemma}

	\begin{proof}
		It is clear that $T$ is closed under inverses and that $0 \in T$.
		Consider any $a,b \in T$. We need to show that $a+b \in T$.

		For $r \in R$ let $g_r \colon H \to R/H$ be given by $g_r(x) \coloneqq rx/H$; this is a group homomorphism.
		Since for all $s \in R$, $[H : \ker(g_s)] \leq [R:H]$ which is finite, we can find a smallest natural number $n$ for which there exists a finite index subgroup $K_n$ of $H$ such that
		\[
		(\forall s \in R) (b \in sK_n/H \Rightarrow [K_n:\ker(g_s) \cap K_n] \leq n).
		\]

		\paragraph{\textbf{Case 1.}} For every finite index subgroup $K$ of $K_n$, there are $r,s \in R$ with $a \in rK/H$ and $b \in sK/H$ such that $\ker(g_r) \cap K$ and $\ker(g_s) \cap K$ are coset-independent subgroups of $K$.

        Then, since $g_r^{-1}(a) \cap K$ and $g_s^{-1}(b) \cap K$ are cosets of  $\ker(g_r) \cap K$ and $\ker(g_s) \cap K$, respectively, they have a non-empty intersection, i.e. there is $k \in K$ with $rk/H=a$ and $sk/H=b$.
		Hence, $a+b = (r+s)k/H \in RK/H$. Since this holds for every finite index subgroup $K$ of $K_n$ (so also for every finite index subgroup of $R$), we conclude that $a+b \in T$.

		\paragraph{\textbf{Case 2.}} There exists a finite index subgroup $K$ of $K_n$ such that for all $r,s \in R$ with $a \in rK/H$ and $b \in sK/H$, $\ker(g_r) \cap K$ and $\ker(g_s) \cap K$ are coset-dependent subgroups of $K$.

		By the definition of $T$, we can pick $r_0 \in R$ with $a \in r_0K/H$. By Remark~\ref{remark: coset independent}, for any $s \in R$ with $b\in sK/H$ (by the definition of $T$, at least one such $s$ exists),
		\[
		\ker(g_{r_0}) \cap K \leq (\ker(g_{r_0}) \cap K) + (\ker(g_s) \cap K) \lneq K.
		\]
		Since $\ker(g_{r_0}) \cap K$ is a fixed finite index subgroup of $K$, there are only finitely many possibilities $L_0,\dots,L_{m-1}$ for $(\ker(g_{r_0}) \cap K) + (\ker(g_s) \cap K)$ when $s$ varies. Also, by the above strict inclusion and the fact that $K \leq K_n$, we have
		\[
		[(\ker(g_{r_0}) \cap K) + (\ker(g_s) \cap K): (\ker(g_s) \cap K)] < [K : \ker(g_s) \cap K] \leq [K_n : \ker(g_s) \cap K_n]
		\]
		for every $s \in R$ with $b\in sK/H$.

		Put $K_{n-1}\coloneqq L_0 \cap \dots \cap L_{m-1}$. This is a finite index subgroup of $H$ contained in $(\ker(g_{r_0}) \cap K) + (\ker(g_s) \cap K)$ for every $s\in R$ with $b\in sK/H$.
		Thus, we conclude that
		\[
		(\forall s \in R) (b \in sK_{n-1}/H \Rightarrow [K_{n-1}:\ker(g_s) \cap K_{n-1}] < [K_n : \ker(g_s) \cap K_n]):
		\]
		indeed, if $b\in sK_{n-1}/H$, then $b\in sK/H$ (because $K_{n-1}\leq K$), so, by the above, we have $[L_i:\ker(g_s) \cap K]< [K_n : \ker(g_s) \cap K_n]$, where $L_i=(\ker(g_{r_0}) \cap K) + (\ker(g_s) \cap K)$, and clearly $[L_i:\ker(g_s) \cap K]\geq [K_{n-1}:\ker(g_s) \cap K_{n-1}]$.

		Therefore, by the choice of $K_n$, we get
		\[
		(\forall s \in R)(b \in sK_{n-1}/H \Rightarrow [K_{n-1}:\ker(g_s) \cap K_{n-1}] \leq n-1),
		\]
		a contradiction with the minimality of $n$.
	\end{proof}

	\begin{theorem}\label{theorem: main theorem}
		Let $R$ be an arbitrary ring $\emptyset$-definable in a structure $M$ and $A \subseteq M$. Then for every $A$-definable finite index subgroup $H$ of $(R,+)$, the set $H+R\cdot H$ contains an $A$-definable, two-sided ideal of $R$ of finite index.
	\end{theorem}

	\begin{proof}
		Without loss of generality we can assume that $M$ is the monster model.
		Define $T$ as in Lemma~\ref{lemma: main lemma}. Let $H'\coloneqq \pi_H^{-1}[T]$, where $\pi_H \colon R \to R/H$ is the quotient map. By Lemma~\ref{lemma: main lemma}, $H'$ is a subgroup of $(R,+)$. We can find a finite index subgroup $H_0$ of $H$ such that $R\cdot H_0/H=T$; equivalently, $H+R\cdot H_0 = H'$.
		For instance, for each $a\in (R/H)\setminus T$, we can choose a finite index $H_a\leq R$ such that $a\notin R\cdot H_a/H$. Then $H_0'\coloneqq H\cap \bigcap_{a\notin T} H_a$ is as stated.

		In fact, $H_0$ can be chosen to be $A$-definable, namely $H_0\coloneqq \bigcap_{r \in R} g_r^{-1}[T]$ (where $g_r \colon H \to R/H$ is defined by $g_r(x):=rx/H$) works: it is a subgroup by Lemma~\ref{lemma: main lemma};
		it has finite index, as it contains the subgroup $H_0'$ defined in the preceding paragraph (since $g_r[H_0']\subseteq T$ for each $r\in R$);
		it is definable over $A$ by the formula $(x \in H) \wedge (\forall r \in R) (rx/H \in T)$ (note that $T$ is finite and $A$-invariant so $A$-definable, as we work in the monster model); $R\cdot H_0/H \supseteq T$ by the definition of $T$ and the fact that the index $[R:H_0]$ is finite; $R\cdot H_0/H \subseteq T$ by the definition of $H_0$.

		Since $(R \cup \{1\}) \cdot H_0$ is generic and symmetric, by Lemma~\ref{lemma: generation in finitely many steps by generic}, the left ideal $(H_0)$ generated by $H_0$ is the $n$-fold sum
		\[
		((R \cup \{1\}) \cdot H_0)^{+n}=(R \cup \{1\}) \cdot H_0+ \dots + (R \cup \{1\}) \cdot H_0
		\]
		for some $n$, so it is $A$-definable. Also, $(H_0) \subseteq H' \subseteq H + R\cdot H$. Since $(H_0)$ is an $A$-definable, left ideal of $R$ of finite index, we finish using
		Fact~\ref{fact: ideal component}.
	\end{proof}

	\begin{corollary}\label{corollary: main corollary}
		Every ring $R$ has property \nref{condition:group_R00}{1\frac12} from Definition~\ref{definition:conditions}, i.e.\ for every $A$-definable, finite index subgroup $H$ of $(R,+)$:
		\begin{enumerate}[label=(\roman*)]
			\item the set $H+R\cdot H$ contains an $A$-definable ideal of $R$ of finite index;
			\item the set $H+H\cdot R$ contains an $A$-definable ideal of $R$ of finite index;
			\item the set $H+(R\cup \{1\}) \cdot H \cdot (R \cup \{1\})$ contains an $A$-definable ideal of $R$ of finite index.
		\end{enumerate}
	\end{corollary}

	\begin{proof}
		Item (i) is Theorem~\ref{theorem: main theorem}, (ii) can be proved analogously, and (iii) follows from (i).
	\end{proof}

	In Example~\ref{example: RHR not enough}, we will see that in item (iii) in that last corollary we cannot write $H+R\cdot H\cdot R$: in this example, $H+R\cdot H\cdot R=H$ does not have a finite index ideal. If one assumes that $R$ is left [or right] s-unital, then one can write $H+R\cdot H\cdot R$ in (iii), because $H\cdot R \subseteq R\cdot H\cdot R$ [resp. $R\cdot H \subseteq R\cdot H\cdot R$].

	As a conclusion of Proposition~\ref{proposition: equivalence of finite numbers of steps} and Corollary~\ref{corollary: main corollary}, we get
	\begin{corollary}\phantomsection\label{corollary: which conditions hold for rings}
		\begin{enumerate}
			\item Conditions \nref{condition:group_R00}{1\frac12}, \nref{condition:X'_subgroup}{1\frac12}, and \nref{condition:X'_R0}{1\frac12} hold for all rings.
			\item All five conditions \nref{condition:X}{1\frac12} -- \nref{condition:X'_R0}{1\frac12} hold for all rings of positive characteristic.
			\item Conditions \nref{condition:def_R00}{2\frac12} and \nref{condition:X}{2\frac12} hold for all s-unital rings.
		\end{enumerate}
	\end{corollary}

	The last two items yield Theorem~\ref{theorem: Main Theorem}, more precisely (using also Corollary~\ref{corollary: 0=00 in rings}):

	\begin{corollary}\label{corollary: 3 steps}
		Let $R$ be a $\emptyset$-definable ring.
		\begin{enumerate}
			\item
			If $R$ is s-unital, then
			\begin{itemize}
				\item
				$\RRaT_A + \RR \cdot \RRaT_A + \RR \cdot \RRaT_A = \RR^{00}_A = \RR^0_A$,
				\item
				$\RRaT_A + \RRaT_A \cdot \RR + \RRaT_A \cdot \RR = \RR^{00}_A= \RR^0_A$,
				\item
				$\RRaT_A + \RR \cdot \RRaT_A \cdot \RR+ \RR\cdot \RRaT_A \cdot \RR= \RR^{00}_A= \RR^0_A$.
			\end{itemize}
			\item
			If $R$ is of positive characteristic (not necessarily s-unital), then
			\begin{itemize}
				\item
				$\RRaT_A + \RR \cdot \RRaT_A = \RR^{00}_A= \RR^0_A$,
				\item
				$\RRaT_A + \RRaT_A \cdot \RR = \RR^{00}_A= \RR^0_A$,
				\item
				$\RRaT_A + (\RR \cup \{1\}) \cdot \RRaT_A \cdot (\RR \cup \{1\})= \RR^{00}_A= \RR^0_A$.
			\end{itemize}
		\end{enumerate}
	\end{corollary}

	\begin{remark}
		\label{remark:3_steps_one_sided}
		Using Remark~\ref{remark:equivalence_one-sided}, we can see that Corollary~\ref{corollary: which conditions hold for rings}(3) and Corollary~\ref{corollary: 3 steps}(1) have obvious analogues in the case when $R$ is only left [or right] s-unital. For instance, if $R$ is left s-unital, then
		\begin{itemize}
			\item
			$\RRaT_A + \RR \cdot \RRaT_A + \RR \cdot \RRaT_A = \RR^{00}_A = \RR^0_A$,
			\item
			$\RRaT_A + (\RR \cup \{1\})\cdot \RRaT_A \cdot (\RR \cup \{1\})+ (\RR \cup \{1\})\cdot \RRaT_A \cdot (\RR \cup \{1\})= \RR^{00}_A= \RR^0_A$.
		\end{itemize}
	\end{remark}

	Thus, by Fact~\ref{fact: conditions equivalent to 000=00} (or Proposition~\ref{proposition: equality of components and fin. many steps}), we obtain

	\begin{corollary}\label{corollary: 0=00=000}
		For an arbitrary $\emptyset$-definable ring $R$ which is (even one-sided) s-unital or of positive characteristic, $\RRiD_A=\RRiT_A = \RRiI_A$.
	\end{corollary}

	\begin{question}\label{question: finitely many steps for all rings}
		Do the last two corollaries (except $\RRiD_A=\RRiT_A$) hold for all rings?
	\end{question}

	\begin{question}\label{question: 2 steps}
		Are 2 or even $1\frac12$ (instead of $2\frac12$) steps enough in
		Corollary~\ref{corollary: 3 steps}(1)?
	\end{question}

	In Examples~\ref{example: Z[X]} and \ref{example:independent_functionals}, we will see that one cannot decrease the number of steps to $1$, even for commutative, unital rings of positive characteristic.

	\section{Commutative, finitely generated, unital rings}\label{section: comm. fin. gen. rings}

	We will give a different (and easier) proof of the following weaker version of Theorem~\ref{theorem: main theorem} for commutative, finitely generated, unital rings.

	\begin{proposition}\label{proposition: Tomek}
		Let $R$ be a $\emptyset$-definable, unital, commutative, finitely generated ring, say generated by $1$ and a set $A$ of cardinality $n$.
		Then for every finite index subgroup $H$ of $(R,+)$, the $(n+1)$-fold sum $(R\cdot H)^{+(n+1)}$ contains an $A$-definable ideal of $R$ of finite index.
	\end{proposition}

	In fact, we prove more:
	\begin{proposition}\label{proposition: Tomek2}
		Let $R$ be a $\emptyset$-definable, unital, commutative, finitely generated ring, say generated by $1$ and a set of cardinality $n$ contained in $A$.
		\begin{enumerate}
			\item Then for every thick subset $D$ of $(R,+)$, $(R\cdot D)^{+(n+1)}$ contains an $A$-definable ideal of $R$ of finite index. Thus, all conditions \nref{condition:X}{n+1} -- \nref{condition:X'_R0}{n+1} from Definition~\ref{definition:conditions} hold over $A$.
			\item If $R$ is additionally of positive characteristic, then we have the same conclusion with $n$ in place of $n+1$.
		\end{enumerate}
	\end{proposition}

	\begin{proof}
		(1) By Proposition~\ref{proposition: equivalence of finite numbers of steps}, the last claim follows from the first one. To show the first claim, we can assume that $A=\{a_1,\dots,a_n\}$ is a set of generators. Since $D$ is thick, for every $i \in \{1,\dots,n\}$ there exist natural numbers $k_i<l_i$ such that $a_i^{l_i} - a_i^{k_i} \in D$;
		for the same reason, there exists a non-zero $k \in \Z$ such that $k\times 1 \coloneqq 1^{+k}\in D$. We claim that the $A$-definable ideal
		\[
		I\coloneqq (k\times 1) +({a_1}^{l_1} -{a_1}^{k_1}) + \dots + ({a_n}^{l_n} - {a_n}^{k_n})
		\]
		is contained in $(R\cdot D)^{+(n+1)}$ and has finite index. Only the latter statement requires an explanation. Since $R$ is a homomorphic image of the ring of polynomials $\Z[X_1,\dots,X_n]$, it is enough to prove it for $R=\Z[X_1,\dots,X_n]$ and $A\coloneqq \{X_1,\dots,X_n$\}. Using an inductive argument and the fact that in the ring of polynomials in one variable over a commutative, unital ring we can divide with reminders by monic polynomials (the monic polynomials that we are using here are $X_i^{l_i} - X_i^{k_i}$), we easily get that every polynomial in $\Z[X_1,\dots,X_n]$ is congruent modulo $I$ to some polynomial in $\Z_k[X_1,\dots,X_n]$ of the $X_i$-degree less than $l_i$ for every $i \in \{1,\dots,n\}$. Hence, $R/I$ is finite.

		(2) If $R$ is of positive characteristic, we do not need the term $(k \times 1)$ in $I$ in the above argument, so $n$ steps is enough.
	\end{proof}

	Note that in the special case of $n=1$, Proposition~\ref{proposition: Tomek2}(1) decreases the number of steps obtained in Corollary~\ref{corollary: 3 steps}(1) [resp. (2)] from $2\frac12$ to $2$ [resp. from $2$ to $1$], but, of course, in a much more narrow class of rings.

	\section{Finitely generated rings}\label{section: fin. gen. rings}

	Here, we give a positive answer to Questions~\ref{question: introduction, finite number of steps for all rings} and \ref{question: finitely many steps for all rings}
	for arbitrary finitely generated rings,
	but with the number of steps dependent on the number of generators and larger than the number given by Corollary~\ref{corollary: 3 steps} (for s-unital or positive-characteristic rings).

	\begin{theorem}\label{theorem: bound for fin gen. rings}
		If $R =\langle r_0,\dots,r_{n-1} \rangle$ is a $\emptyset$-definable, finitely generated ring, then \nref{condition:X}{n+1+\frac12} holds for any $A$ containing $r_0,\dots, r_{n-1}$.
		In fact, it is enough to assume that $R = (R +\Z)r_0 + \dots + (R + \Z)r_{n-1}$ and $Rr_0 + \dots + Rr_{n-1}=r_0R + \dots + r_{n-1}R$.
	\end{theorem}

	Note that indeed $R =\langle r_0,\dots,r_{n-1} \rangle$ implies the assumptions after ``In fact'', and that for a commutative $R$ the last assumption trivially holds.

	Let $R$ be as above (i.e.\ satisfies the two assumptions after ``In fact''). Put $S\coloneqq Rr_0 + \dots + Rr_{n-1}$. This is a two-sided ideal of $R$.

	\begin{lemma}
		Let $D$ be an $A$-definable, thick subset of $R$ such that $\RRaT_A \subseteq \bar D$. Then $D + (R\cdot D)^{+(n+1)}$ contains an $A$-definable, finite index ideal of $S$.
	\end{lemma}

	\begin{proof}
		For $i=0,\dots,n-1$, let $k_i \in \mathbb{N} \setminus \{0\}$ be such that $k_ir_i \in D$ (by thickness). Then for $k\coloneqq k_0 \dots k_{n-1}$ we have that $kS \subseteq (R\cdot D)^{+n}$, because every $s \in S$ can be written as $\sum_i r_i' r_i$ (for some $r_i' \in R$) and then $ks =\sum_i k r_i' r_i = \sum_i \frac{k}{k_i}r_i' \cdot k_ir_i \in (R\cdot D)^{+n}$.

		On the other hand, $S/kS$ is of positive characteristic and $(S \cap D)/kS$ is an $A$-definable, thick subset of $S/kS$ such that $(\bar S \cap \bar D)/k\bar S$ contains $(\bar S \cap \RRaT_A)/k\bar S \geq (\bar S,+)^{00}_A/k\bar S = (\bar S/k\bar S,+)^{00}_A$.
		Hence, as \nref{condition:def_R00}{1\frac12} holds for rings of positive characteristic (by Corollary~\ref{corollary: which conditions hold for rings}(2)), we get that $(S \cap D)/kS + S/kS \cdot (S \cap D)/kS$ contains an $A$-definable, finite index ideal of $S/kS$. Then the preimage of this ideal by the quotient map $S \to S/kS$ is an $A$-definable, finite index ideal of $S$ contained in $(S \cap D) + S \cdot (S \cap D) +kS \subseteq D + S \cdot D + kS \subseteq D + (R\cdot D)^{+(n+1)}$.
	\end{proof}
	Since any finite index, $A$-definable ideal of $S$ contains $\bar S^{0}_A$, we obtain the following by compactness.
	\begin{corollary}\label{corollary: inclusion}
		$\RRaT_A+ \bar S^{0}_A \subseteq \RRaT_A + (\bar R \cdot \RRaT_A)^{+(n+1)}$.
	\end{corollary}

	\begin{lemma}
		$\bar S^0_A$ is a left ideal of $\bar R$.
	\end{lemma}

	\begin{proof}
		Recall that $M$ is a structure in which $R$ is $\emptyset$-definable. Let $M'$ be the expansion of $M$ by the additional sort $(\Z,+)$ and the natural action of this sort on $R$. We choose a monster model $\bar M'$ so that $\bar M$ (i.e.\ the interpretation of $M$ in $\bar M'$) is a monster model of $\Th(M)$. All the components and definable subsets of $\bar R$ will be in the sense of the original language of $M$.

		Consider any $r \in \bar R$. We need to show that $r\bar S^0_A \subseteq \bar S^0_A$. By the assumption on $R$, $r \in \bar S + \sum_i \bar \Z r_i$. Since $\bar S^0_A$ is an ideal of $\bar S$, we have $\bar S\cdot \bar S^0_A \subseteq \bar S^0_A$, so it
		suffices to show (1) and (2), where:
		\begin{enumerate}[label=(\arabic*), nosep]
			\item $\bar \Z\cdot \bar S^0_A \subseteq \bar S^0_A$,
			\item $r_i \bar S^0_A \subseteq \bar S^0_A$ for all $i$.
		\end{enumerate}

		Ad 1. $\bar S^0_A = \bigcap_j \bar I_j$ for some $A$-definable, finite index ideals $I_j$ of $S$. Clearly $\Z \cdot I_j \subseteq I_j$, so $\bar \Z \cdot \bar I_j \subseteq \bar I_j$. This implies that $\bar \Z \cdot \bar S^0_A \subseteq \bar S^0_A$.

		Ad 2.
		As $r_i \in A$, the map $f_{r_i} \colon S \to S$ given by $f_{r_i}(x)\coloneqq r_ix$ is an $A$-definable group homomorphism. Hence, one easily gets $r_i (\bar S,+)^0_A \subseteq (\bar S,+)^0_A \subseteq \bar S^0_A$. Also, since \nref{condition:X'_R0}{1\frac12} holds for all rings (by Corollary~\ref{corollary: which conditions hold for rings}(1)), we have $\bar S^0_A = (\bar S,+)^0_A + \bar S \cdot (\bar S,+)^0_A$. Using these two observations, we get $ r_i \bar S^0_A = r_i ( (\bar S,+)^0_A + \bar S \cdot (\bar S,+)^0_A) = r_i (\bar S,+)^0_A + r_i \bar S \cdot (\bar S,+)^0_A \subseteq \bar S^0_A + \bar S \cdot (\bar S,+)^0_A \subseteq \bar S^0_A + \bar S \cdot \bar S^0_A = \bar S^0_A$.
	\end{proof}

	\begin{lemma}\label{lemma: ideal in R}
		$\RRaT_A+ \bar S^{0}_A$ is a left ideal of $\bar R$.
	\end{lemma}

	\begin{proof}
		First, note that $r_0\RRaT_A + \dots + r_{n-1}\RRaT_A = (\bar S,+)^{00}_A$.
		This follows easily from the observations that $r_i\RRaT_A = (r_i\bar R,+)^{00}_A$ and $r_0 \bar R + \dots +r_{n-1} \bar R = \bar S$ (the last equality follows from the last assumption of Theorem~\ref{theorem: bound for fin gen. rings}). Now, consider any $r \in \bar R$, and take the notation described at the beginning of the previous lemma. By that lemma, it remains to show that $r\RRaT_A \subseteq \bar S^0_A$. By the assumption on $R$, $r \in \sum_i \bar \Z r_i + r_i'r_i$ for some $r_i' \in \bar R$. Therefore, $r\RRaT_A \subseteq \sum_i \bar \Z r_i \RRaT_A +r_i'r_i \RRaT_A \subseteq \sum_i \bar \Z \cdot (\bar S,+)^{00}_A + r_i' (\bar S,+)^{00}_A \subseteq \sum_i \bar \Z \cdot \bar S^{0}_A + r_i' \bar S^{0}_A \subseteq \bar S^0_A$, where the last inclusion follows from the previous lemma and item (1) in its proof.
	\end{proof}

	The next corollary implies Theorem~\ref{theorem: bound for fin gen. rings}. More precisely, it implies the left version of Theorem~\ref{theorem: bound for fin gen. rings} (i.e.\ condition \nref{condition:X}{n+1+\frac12} for $i= {\leftarrow}$), but the right version is analogous, and the two-sided version follows from the left version.

	\begin{corollary}\label{corollary: implies theorem bound for fin gen. rings}
		We have $\RRaT_A + (\bar R \cdot \RRaT_A)^{+(n+1)} = \RRaT_A+ \bar S^{0}_A$.
	\end{corollary}

	\begin{proof}
		The inclusion $(\supseteq)$ is Corollary~\ref{corollary: inclusion}. For the opposite inclusion, note that the left hand side is contained in $J =\bar R^{000}_A$. On the other hand, by Lemma~\ref{lemma: ideal in R}, the right hand side is a left ideal of $R$, which is clearly $A$-type-definable and of bounded index, so it contains $\bar R^{000}_A$. Therefore, we get $(\subseteq)$.
	\end{proof}

	\section{Topological rings}\label{section: topological rings}

	Throughout this section, $R$ is a topological ring. Equip it with the full structure (e.g.\ by adding predicates for all subsets of all finite Cartesion powers). In Subsection 3.5 of \cite{GJK22}, the following connected component was defined:

	\[
	\RR^{00}_{\topo}\coloneqq \RR^{00}_R + \mu = \RR^{00}_\emptyset + \mu,
	\]
	where $\mu$ stands for the infinitesimals. What makes this object interesting is \cite[Proposition 3.32]{GJK22} which says that the quotient map $R \to \RR/\RR^{00}_{\topo}$ is the Bohr (i.e.\ universal) compactification of the topological ring $R$. Another point is that Theorem~\ref{theorem: main topological}, which we will prove below, implies directly that condition ($\dagger \dagger$) in Subsection 4.5 of \cite{GJK22} holds for all unital, topological rings, and hence (as stated right below ($\dagger \dagger$) in \cite{GJK22}) one gets the simplified formulas for the Bohr compactifications of the topological groups $\UT_n(R)$ and $\T_n(R)$ over any unital, topological ring $R$ (analogous to those in \cite[Corollary 4.5]{GJK22}).

	The above ring component is an analog of the component ${\bar G}^{00}_{\topo}\coloneqq {\bar G}^{00}_\emptyset + \mu$ of a topological group $G$ (equipped with the full structure), which was studied in \cite{GPP14} and \cite{KrPi}. It was introduced because of the same reason as the above component of a topological ring, namely: the quotient map $G \to {\bar G}^{00}_{\topo}$ is the Bohr compactification of the topological group $G$. And in \cite[Subsection 4.5]{GJK22}, this was used to find formulas for Bohr compactifications of $\UT_n(R)$ and $\T_n(R)$ over unital, topological rings.

	Returning to the topological ring $R$, note that $(R,+)$ is a topological group, so we have
	\[
	\RRaT_{\topo}\coloneqq \RRaT_{\emptyset} + \mu.
	\]

	The following fact (see \cite[Lemma 3.29]{GJK22}) will be important in the proof of Theorem~\ref{theorem: topological case} below.

	\begin{fact}\label{fact: R_topo is an ideal}
		$\RR^{00}_{\topo}$ is a two-sided ideal.
	\end{fact}

	Analogously to \cite[Subsection 3.5]{GJK22}, one can define
	\[
	\RR^{0}_{\topo}\coloneqq \RR^{0}_R + \mu = \RR^{0}_\emptyset + \mu,
	\]
	and show that it is a two-sided ideal and the quotient map $R \to \RR/\RR^{0}_{\topo}$ is the universal profinite compactification of the topological ring $R$.

	By Corollary~\ref{corollary: 0=00 in rings}, we immediately get

	\begin{corollary}\label{corollary: R00_topo = R0_topo}
		If $R$ is s-unital or of positive characteristic, then $\RR^{00}_{\topo} = \RR^{0}_{\topo}$.
	\end{corollary}


	Using this together with Corollary~\ref{corollary: 3 steps}, we will easily prove the topological counterpart of the latter, i.e.\ Theorem~\ref{theorem: main topological}.

	\begin{theorem}\label{theorem: topological case}
		Let $R$ be a topological ring.
		\begin{enumerate}
			\item
			If $R$ is s-unital, then
			\begin{itemize}
				\item
				$\RRaT_{\topo} + \RR \cdot \RRaT_{\topo} + \RR \cdot \RRaT_{\topo} = \RR^{00}_{\topo} = \RR^{0}_{\topo}$,
				\item
				$\RRaT_{\topo} + \RRaT_{\topo} \cdot \RR + \RRaT_{\topo} \cdot \RR = \RR^{00}_{\topo}= \RR^{0}_{\topo}$,
				\item
				$\RRaT_{\topo} + \RR\cdot \RRaT_{\topo} \cdot \RR+ \RR\cdot \RRaT_{\topo} \cdot \RR= \RR^{00}_{\topo}= \RR^{0}_{\topo}$.
			\end{itemize}
			\item
			If $R$ is of positive characteristic (not necessarily s-unital), then
			\begin{itemize}
				\item
				$\RRaT_{\topo} + \RR \cdot \RRaT_{\topo} = \RR^{00}_{\topo}= \RR^{0}_{\topo}$,
				\item
				$\RRaT_{\topo} + \RRaT_{\topo} \cdot \RR = \RR^{00}_{\topo}= \RR^{0}_{\topo}$,
				\item
				$\RRaT_{\topo} + (\RR \cup \{1\}) \cdot \RRaT_{\topo} \cdot (\RR \cup \{1\})= \RR^{00}_{\topo}= \RR^{0}_{\topo}$.
			\end{itemize}
		\end{enumerate}
	\end{theorem}

	\begin{proof}
		The equality $\RR^{00}_{\topo} = \RR^{0}_{\topo}$ follows from Corollary~\ref{corollary: R00_topo = R0_topo}. Let us prove the first equality in item (1). All the rest is analogous.

		We have $\RRaT_{\topo} + (\bar R \cdot \RRaT_{\topo})^{+2} = \RRaT_{\emptyset}+\mu + (\bar R \cdot (\RRaT_{\emptyset}+\mu))^{+2}$ which is contained in $\RRaT_{\emptyset} + (\bar R \cdot \RRaT_{\emptyset})^{+2} + \mu + (\bar R \mu)^{+2} = \bar R^{00}_\emptyset + \mu + (\bar R \mu)^{+2} = \bar R^{00}_\emptyset + \mu$
		(with the first equality following from Corollary~\ref{corollary: 3 steps}, and the second one from Fact~\ref{fact: R_topo is an ideal})
		and which contains $\RRaT_{\emptyset} + (\bar R \cdot \RRaT_{\emptyset})^{+2} + \mu = \bar R^{00}_\emptyset + \mu$ (where the equality follows from Corollary~\ref{corollary: 3 steps}). Therefore, $\RRaT_{\topo} + (\bar R \cdot \RRaT_{\topo})^{+2} = R^{00}_\emptyset + \mu = \RR^{00}_{\topo}$.
	\end{proof}

	The following question is related to Theorem~\ref{theorem: topological case}.
	\begin{question}\label{question: R mu}
		Does $(\bar R \cup \{1\}) \mu$ generate a group in finitely many steps?
	\end{question}

	Recall from \cite{KrPi} that for a topological group $G$ (equipped with the full structure), ${\bar G}^{000}_{\topo}\coloneqq {\bar G}^{000}_\emptyset \langle \mu^{\bar G} \rangle$.
	Analogously, for a topological ring $R$ we can define
	\[
	\RR^{000}_{\topo}\coloneqq \RR^{000}_\emptyset + \langle (\bar R \cup \{1\}) \mu (\bar R \cup \{1\})\rangle.
	\]
	This is the smallest bounded index, invariant, two-sided ideal of $\bar R$ containing $\mu$.\footnote{Note that, in contrast to $\RR^{00}_{\topo}$ and $\RR^{0}_{\topo}$, we cannot define $\RR^{000}_{\topo}$ as $\RR^{000}_\emptyset+\mu$, since we do not know whether the latter is necessarily an ideal.}

	Let $J_{\topo}$ be the subgroup of $(\RR,+)$ generated by $(\RR \cup \{1\}) \cdot \RRaT_{\topo} (\RR \cup \{1\})$. Since by \cite[Corollary 2.37]{HKP1} we know that $\RRaT_{\topo} = (\RR,+)^{000}_{\topo}$, the proof of Fact~\ref{fact: occurrence of 000} adapts to yield that $J_{\topo}=\RR^{000}_{\topo}$. Thus, Fact~\ref{fact: conditions equivalent to 000=00} has its obvious topological counterpart.
	Hence, in the same way as in Corollary~\ref{corollary: 0=00=000}, Theorem~\ref{theorem: topological case} implies

	\begin{corollary}\label{corollary: 0=00=000 for top. rings}
		For an arbitrary topological ring $R$ which is s-unital or of positive characteristic, $\RR^{0}_{\topo}=\RR^{00}_{\topo} = \RR^{000}_{\topo}$.
	\end{corollary}

	Alternatively, one can see $\RR^{00}_{\topo} = \RR^{000}_{\topo}$ as a consequence of Corollary~\ref{corollary: 0=00=000} and Fact~\ref{fact: R_topo is an ideal}.

	\section{Examples}\label{section: examples}
	We study several classical examples, as well as construct some new ones, for which we compute the number of steps (i.e.\ $n$) needed in conditions from Definition~\ref{definition:conditions}, which show optimality of some of our results and answer several natural questions discussed in the introduction.
	In all the examples considered in this section, we take $M=R$, equipped with the \textbf{full structure}, and we take $A=\emptyset$.\footnote{Since we have the full structure on $M$, the choice of $A$ is essentially immaterial.}

	One should keep in mind the relationships obtained in Proposition~\ref{proposition: equivalence of finite numbers of steps}, which
	will be often used implicitly.

	By the examples from Subsections 3.3 and 4.4 of \cite{GJK22}, one gets:
	\begin{enumerate}
		\item for $R\coloneqq \mathbb{Z}$, conditions \nref{condition:group_R00}{\frac12} and \nref{condition:X}{1} hold, but \nref{condition:X}{\frac12} fails (see Example 3.18 and Lemma 4.17 in \cite{GJK22}); thus, \nref{condition:def_R00}{1} holds, while \nref{condition:def_R00}{\frac12} fails;
		\item for $R\coloneqq K$ or $R\coloneqq K[\bar X]$, where $K$ is an infinite field and $\bar X$ is a tuple of possibly infinitely many variables, conditions \nref{condition:X}{1}, \nref{condition:def_R00}{1}, and \nref{condition:group_R00}{1} hold by \cite[Proposition 4.21]{GJK22}, but for example in positive characteristic \nref{condition:X}{\frac12}, \nref{condition:def_R00}{\frac12}, and \nref{condition:group_R00}{\frac12} fail (see Example 3.23 in \cite{GJK22}).
	\end{enumerate}

	\subsection{Example: the ring \texorpdfstring{$\pmb{\Z[X]}$}{Z[X]}}
	\label{example: Z[X]}
	Our first goal is to analyze the ring of polynomials $\Z[X]$.

	Let $R\coloneqq \Z[X]$. By Proposition~\ref{proposition: Tomek2}, all conditions \nref{condition:X}{2} -- \nref{condition:X'_R0}{2} hold. We will show that \nref{condition:group_R00}{1} fails; thus, all other conditions fail for $n=1$.
	This shows that one can not decrease the number of steps from $1\frac12$ to $1$ in Corollary~\ref{corollary: main corollary} (in particular, in Theorem~\ref{theorem: main algebraic theorem}, thus showing that it is optimal in the sense of the number of steps) and from $2\frac12$ to $1$ in Corollary~\ref{corollary: 3 steps}(1).

	The fact that \nref{condition:group_R00}{1} fails follows immediately from

	\begin{lemma}
		\label{lemma:Z[x]_no_iii1}
		The group $H\coloneqq 2R+\langle X^n-X^m : n,m \text{ are both prime or both not prime}\rangle\leq (R,+)$ has index at most $4$ in $R$, but $R\cdot H$ does not contain any finite index ideal of $R$.
	\end{lemma}

	\begin{proof}
		It is easy to see that $[R:H] \leq 4$, since every element of $R/H$ is represented by one of $0,1,X^2,X^2+1$. Namely, for any $P(X)=\sum_j a_jX^j\in \Z[x]$ we have $P(X)\equiv \alpha+\beta X^2 \pmod H$, where
		\begin{align*}
			\alpha = \sum_{j\textrm{ nonprime}} a_j \mod 2, && \beta= \sum_{j\textrm{ prime}}a_j\mod 2.
		\end{align*}
		(In fact, by the argument in Claim 4 below, we get $[R:H]=4$.)

		For each positive integer $m$ and prime $p\equiv -1\pmod {m!}$, consider the polynomial
		\[
		Q_{m,p}(X)\coloneqq X^{m!+m}+pX^m+m!p.
		\]
		Note that $Q_{m,p}$ is irreducible in $R$ by Eisenstein's criterion: since $p\equiv -1 \pmod {m!}$, $p$ does not divide $m!$, so $p^2$ does not divide $m!p$, and, on the other hand, $p$ divides all coefficients of $Q_{m,p}$ except for the leading coefficient. Note also that for each $m$, a prime $p\equiv -1 \pmod{m!}$ exists by Dirichlet's theorem on arithmetic progressions.

		\begin{clm}
			If $I$ is an ideal of finite index
			at most $m$ in $R$, then the ideal $(m!,X^{m!+m}-X^m)$ is contained in $I$.
		\end{clm}
		\begin{clmproof}
			Since $[R:I]\leq m$, we clearly have $m! \in I$.
			We also get that that $X^j - X^i \in I$ for some $i<j$ from $\{0,\dots,m\}$. Then $X^{m+k} - X^m \in I$ for $k\coloneqq j-i$. Therefore, $X^{m!+m}-X^m=(1+X^k+X^{2k}+\ldots+X^{m!-k})(X^{m+k}-X^m) \in I$.
		\end{clmproof}
		Now, write $Q_m'$ for the polynomial $X^{m!+m}-X^m$.

		\begin{clm}
			If $I$ is an ideal of finite index in $R$, then for sufficiently large $m$, for all $p\equiv -1 \pmod {m!}$ we have $Q_{m,p}\in I$.
		\end{clm}
		\begin{clmproof}
			By Claim~1, we have that for any $m\geq [R:I]$,
			\[
			I\supseteq I_m\coloneqq (m!,X^{m!+m}-X^m).
			\]
			On the other hand, if $p\equiv -1 \pmod{m!}$, then it is easy to see that $Q_{m,p}\equiv Q'_{m}\equiv 0\pmod{I_m}$.
		\end{clmproof}
		\begin{clm}
			For all $m\geq 2$ we have $Q_{m,p}\equiv Q_m'\pmod{H}$.
		\end{clm}
		\begin{clmproof}
			Fix any $m\geq 2$ and $p\equiv -1\pmod{m!}$. Since $m\geq 2$, it follows that $2$ divides $m!$, so $p\equiv -1\pmod{2}$ and $m!\equiv 0 \pmod{2}$. Since $H$ contains $2R$, the conclusion follows.
		\end{clmproof}

		\begin{clm}
			If $n$ is prime and $m$ is not prime, then $X^n-X^m\notin H$ and $1 \notin H$.
		\end{clm}
		\begin{clmproof}
			This follows from the observation that a polynomial belongs to $H$ if and only if after reducing its coefficients modulo 2, both the number of monomials $X^i$ with $i$ prime and the number of monomials $X^i$ with non-prime $i$ are even.
		\end{clmproof}
		Now, let $m$ and $p\equiv -1\pmod{m!}$ be arbitrary primes. Then $m!+m$ is clearly not prime, so, by Claim 4, $Q_m'\notin H$, and hence, by Claim 3, $Q_{m,p}\notin H$. Since $Q_{m,p}$ is irreducible and $1\notin H$, it follows that $Q_{m,p}\notin R\cdot H$. By infinitude of primes and Claim 2, it follows that $R\cdot H$ does not contain any ideal of finite index.
	\end{proof}

	\subsection*{More variables and free rings}
	Let us briefly look at $R\coloneqq \Z[X,Y]$. Since there is an epimorphism $\Z[X,Y] \to \Z[X]$, by
	Lemma~\ref{lemma:Z[x]_no_iii1}, we get that \nref{condition:group_R00}{1} fails, while \nref{condition:group_R00}{1\frac12} holds (as it holds for every ring). By Corollary~\ref{corollary: which conditions hold for rings}, \nref{condition:X}{2\frac12} and \nref{condition:def_R00}{2\frac12} hold.
	The question whether \nref{condition:X}{2} holds remains open.
	The same comments and question are valid for $R\coloneqq \Z[\bar X]$ (where $\bar X$ is a possibly infinite tuple of variables). In order to answer Question~\ref{question: 2 steps} in the affirmative for unital, commutative rings, it is enough to restrict the situation to the rings $\Z[\bar X]$ (for arbitrarily long $\bar X$), i.e.\ determine whether \nref{condition:X}{2} or \nref{condition:X}{1\frac12} holds for such rings. This follows easily from the existence of an epimorphism
	$\Z[\bar X] \to R$ (where $R$ is an arbitrary commutative, unital ring; see the comments after Question 4.9 in \cite{GJK22}).
	A similar comment applies to Question~\ref{question: 2 steps} for arbitrary rings, but using the free rings in non-commuting variables in place of $\Z[\bar X]$.
	Summarizing, Question~\ref{question: 2 steps} remains open and it is enough to answer it for free rings:
	as we have seen above, the answer is positive for $\Z[X]$ (at least in the sense that \nref{condition:X}{2} holds; we do not know whether \nref{condition:X}{1\frac12} holds), but already for $\Z[X,Y]$ we do not know the answer.

	\subsection{Example: an exotic ring of characteristic 2}
	\label{example:independent_functionals}
	The next example shows that \nref{condition:group_R00}{1} may fail even if $R$ is unital, commutative, and of positive characteristic, hence $1\frac12$ steps in Corollary~\ref{corollary: main corollary} (in particular, in Theorem~\ref{theorem: main algebraic theorem}) is an optimal bound even for such rings, and so $1\frac12$ steps is also optimal in Corollary~\ref{corollary: 3 steps}(2)

	Let $G\coloneqq \Z_2^{\oplus \omega}$. Put $U_n\coloneqq \{\eta \in G: \eta(n)=0\}$ for $n<\omega$. The subgroups $(U_n)_{n<\omega}$ are coset-independent subgroups of index 2 in $G$
	(in the sense that any intersection of cosets of $U_n$, $n \in \omega$, is non-empty, cf.\ Definition~\ref{def:coset_indep}).
	Let $\bar G \succ G$ be a monster model in the pure group language expanded by unary relational symbols for $U_0,U_1,\dots$. Then $(\bar U_n)_n$ are coset independent in $\bar G$, and of index $2$.
	So the linear functionals $f_n \colon \bar G \to \Z_2$ such that $\bar U_n =\ker(f_n)$ for all $n<\omega$ are linearly independent.
	Denote the elements of $\Z_2$ by $0, e$.

	Put $R\coloneqq \bar G \oplus \{0,e\} \oplus \{0,1\}$. We will define a multiplication on $R$ turning it into a unital, commutative ring (with unit 1) in which the finite index subgroup $\bar G$ has the property that $R\bar G$ does not contain an ideal of finite index, i.e.\ \nref{condition:group_R00}{1} fails.

	Note that $\bigcap_{n<\omega} \bar U_n$ is of large cardinality, at least the degree of saturation of $\bar G$.
	Choose a basis $(r_i)_{i<\theta}$ of $\bigcap_{n<\omega} \bar U_n$; then $\theta$ is at least the degree of saturation of $\bar G$.

	Extend $(r_i)_{i < \theta}$ to a basis $(r_i)_{i <\lambda}$ of $\bar G$. First, define $f \colon \{r_i: i<\lambda\} \times \{r_i : i <\lambda\} \to \Z_2$ by
	\[
	f(r_i,r_j)\coloneqq
	\left\{ \begin{array}{ll}
		f_i(r_j) & i<\omega\\
		f_j(r_i) & j < \omega\\
		0 & \textrm{otherwise}.
	\end{array} \right.
	\]

	Since $f_i(r_j)=0$ whenever $i,j<\omega$, $f$ is well-defined. Now, extend it to a function (also denoted by $f$) from $(\{r_i: i<\lambda\} \cup \{e,1\}) \times (\{r_i : i <\lambda\} \cup \{e,1\}) \to \{0,e,1\}$ by:
	\[
	f(r_i,e) = f(e,r_i) = f(e,e)=0, f(r_i,1)=f(1,r_i)=r_i, f(e,1)=f(1,e)=e, f(1,1)=1.
	\]

	Finally, extend this extended $f$ to a unique bilinear function $\cdot \colon R \times R \to \{0,e\} \oplus \{0,1\}$.
	It is easy to check that $\cdot$ is symmetric, associative, and for every $i<\omega$ and $r \in \bar G$, $r_i \cdot r= f_i(r)$,
	so $(R,+,\cdot)$ is a commutative $\Z_2$-algebra, and in particular a ring.

	\begin{lemma}
		\label{lem:exotic_no_iii}
		For the ring $R$ defined above, \nref{condition:group_R00}{1} fails.
	\end{lemma}
	\begin{proof}
		It easily follows from the definitions that
		\begin{equation}
			\label{eq:indep_funct_RG}
			\tag{$*$}
			R \cdot \bar G = \bar G \sqcup \left(e +\bigcup_{r \in \bar G} \bar G \setminus \ker(r \cdot)\right) \sqcup \{e\}.
		\end{equation}

		It remains to show that $R \cdot \bar G$ does not contain an ideal of finite index.
		For that consider any finite index ideal $I$ of $R$. Then $\lin(r_i: i < \omega) \cap I \ne \{0\}$; so take $i_0,\dots,i_{m-1} <\omega$ with $r\coloneqq r_{i_0} +\dots + r_{i_m-1}$ in this intersection. Since $f_{i_0},\dots,f_{i_{m-1}}$ are linearly independent, $g\coloneqq f_{i_0}+\dots +f_{i_{m-1}}$ is non-zero, so we can pick $s \in \bar G \setminus \ker(g)$. Since $g=r\cdot$, we get that $r\cdot s \ne 0$ which implies that $s \cdot r = r \cdot s = e$. Hence, $e \in I$ (as $r \in I$).

		Recall that $(r_i)_{i<\theta}$ is a basis of $\bigcap_n \bar U_n$, and $\theta > \omega + \omega$. Thus, $\lin(r_i: \omega \leq i < \theta) \cap I \ne \{0\}$, so
		we can choose a nonzero $t$ from this intersection.

		By the last two paragraphs, $e+t \in I$. On the other hand, by the construction of $\cdot$, we see that
		$\lin(r_i: \omega \leq i < \theta) \subseteq \bigcap_{r \in \bar G} \ker(r \cdot)$,
		so $t \notin \bigcup_{r \in \bar G} \bar G \setminus \ker(r \cdot)$. Since also $t \ne 0$, by \eqref{eq:indep_funct_RG}, we get $e+t \notin R \cdot \bar G$. We have proved that $e+t \in I \setminus (R \cdot \bar G)$, so $I$ is not contained in $R \cdot \bar G$.
	\end{proof}

	\subsection{Example: the ring \texorpdfstring{$\pmb{X\Z[X]}$}{XZ[X]}}
	\label{example: XZ[X]}
	Let $R$ be the non-s-unital ring $X\Z[X]$. By Theorem~\ref{theorem: bound for fin gen. rings}, condition \nref{condition:X}{2\frac12} holds (although we do not know whether \nref{condition:X}{2} holds). Of course, \nref{condition:group_R00}{1\frac12} holds,
	and we will adapt the proof of Lemma~\ref{lemma:Z[x]_no_iii1} to show that \nref{condition:group_R00}{1} fails.
	But the most important new thing in comparison with the previous examples is
	that  (as we will soon see)
	$\RRiT_{\emptyset} \ne \RRiD_{\emptyset}$, and so \nref{condition:def_R00}{n} fails for all $n$, although $R$ does not have zero divisors. This shows that the additional assumptions about $R$ in items (3) and (4) of Proposition~\ref{proposition: equivalence of finite numbers of steps}
	cannot be dropped (not even for commutative rings without zero divisors).

	\begin{lemma}
		\nref{condition:def_R00}{n} fails for all $n$.
	\end{lemma}
	\begin{proof}
		Fix $n$.
		By Fact~\ref{fact:bohrIfOnly}, $(\bar \Z,+)^{00}_\emptyset \ne (\bar \Z,+)^0_\emptyset$. So there is a thick subset $F$ of $\Z$ for which $(\bar \Z,+)^{00}_\emptyset \subseteq \bar F$ and such that the $n$-fold sum $F^{+n}$ does not contain a finite index subgroup. Put $D\coloneqq FX + XR$. It is clear that $D$ is thick in $R$ and $\bar D$ contains $\RRaT_\emptyset$, because $\bar D = \bar F X + X\bar R \supseteq (\bar \Z,+)^{00}_\emptyset X + X\bar R$ which is $\emptyset$-type-definable of bounded index.
		On the other hand, for $i \in \{\leftarrow,\rightarrow,\leftrightarrow\}$ we have $D_i = D$, so $D_i^{+n}=D^{+n} = F^{+n}X +XR$, which does not contain a finite index subgroup of $(R,+)$ (because $F^{+n}$ does not contain a finite index subgroup of $(\Z,+)$). So \nref{condition:def_R00}{n} fails.
	\end{proof}

	It is clear that the above lemma remains true for $XS[X]$, where $S$ is any commutative ring of characteristic $0$ (because we only need to know that $S$ satisfies $(\bar S,+)^{00}_\emptyset \ne (\bar S,+)^{0}_\emptyset$).

	A more succinct way to show the lemma above is as follows.

	Note that $I\coloneqq (\bar\Z,+)^{00}X+X\RR$ is a bounded index, $\emptyset$-type-definable ideal of $\RR$, and $(\RR/I,+) \cong \bar \Z/(\bar \Z,+)^{00}$ is not profinite, and thus neither is $\RR/\RRiT_{\emptyset}$. Therefore, $\RRiT_{\emptyset} \ne \RRiD_{\emptyset}$, and so \nref{condition:def_R00}{n} fails for all $n$ by Proposition~\ref{proposition: equality of components and fin. many steps}(1).

	We have also shown that $R$ is an example of a finitely generated, commutative ring with no zero divisors whose Bohr compactification $\RR/\RRiT_\emptyset$ is not profinite.

	\begin{lemma}
		\nref{condition:group_R00}{1} fails.
	\end{lemma}

	\begin{proof}
		Similarly to the analysis of $\Z[X]$ in Lemma~\ref{lemma:Z[x]_no_iii1}, consider
		\[
		H\coloneqq 2R+\langle X^n-X^m : n,m>0 \textrm{ are both prime or both not prime}\rangle.
		\]
		As there, $H$ is a subgroup of $(R,+)$ with $[R:H]=4$.

		Consider the polynomials $Q''_m(X)\coloneqq X^{m!+m}-X^m+m!X$, and fix a finite index ideal $I\unlhd R$.
		Arguing as in Claims~1 and 2 in the proof of Lemma~\ref{lemma:Z[x]_no_iii1}, for any $m\geq [R:I]$ we have $Q''_m(X)\in I$. On the other hand, $Q''_m(X)\notin R\cdot R$, so $Q''_m(X)\notin R\cdot H$. Also, if $m$ is prime, then (arguing as in Claims~3 and 4 in the proof of Lemma~\ref{lemma:Z[x]_no_iii1}) $Q''_m(X)\notin H$. Therefore, $Q''_m(X)\notin H\cup R\cdot H=(R\cup\{1\})\cdot H$. Since there are arbitrarily large primes, it follows that we have a prime $m$ such that $Q''_m(X)\in I\setminus ((R\cup \{1\})\cdot H)$.
	\end{proof}

	\subsection{Example: a nilpotent ring}
	\label{example: RHR not enough}
	The next example shows that in Corollary~\ref{corollary: main corollary}(iii) one cannot replace $R \cup \{1\}$ by $R$, and in the last item of Corollary~\ref{corollary: 3 steps}(2) one can not replace $\bar R \cup \{1\}$ by $\bar R$.

	Let $R$ be the free commutative, nilpotent ring of nilpotency class 3 and of characteristic 2 in infinitely many generators $X_0,X_1,\dots$. (Recall that a ring is nilpotent of class 3 if the product of any three elements equals $0$.) The elements of $R$ are of the form
	\[
	\sum_{i \geq 0} a_iX_i + \sum_{i \geq 0} a_{0i}X_0X_i + \sum_{i \geq 1}a_{1i}X_1X_i + \sum_{i \geq 2}a_{2i}X_2X_i + \dots
	\]
	with all coefficients in $\Z_2$.
	Below $2^j\N_{>0}$ denotes the positive integers divisible by $2^j$.
	Define a group homomorphism $h \colon (R,+) \to \Z_2$ by
	\[
	h\left( \sum_{i \geq 0} a_iX_i + \sum_{i \geq 0} a_{0i}X_0X_i + \sum_{i \geq 1}a_{1i}X_1X_i + \dots\right) = \sum_{i \geq 0} a_{2i+1} + \sum_{j \geq 0} \sum_{i \in 2^j\N_{>0}} a_{ji}.
	\]
	Let $H\coloneqq \ker(h)$, a subgroup of $(R,+)$ of index 2.

	\begin{lemma}
		\label{lemma:nilpotent_RHR}
		$H$ does not contain any ideal of finite index.
	\end{lemma}

	Before the proof, let us see that this lemma yields the desired properties postulated
	above.
	By 3-nilpotency, $H + R\cdot H\cdot R = H$ and, by Lemma \ref{lemma:nilpotent_RHR}, this group does not contain a finite index ideal, so in Corollary~\ref{corollary: main corollary}(iii) one cannot replace $R \cup \{1\}$ by $R$. Let us explain how it implies that in the last item of Corollary~\ref{corollary: 3 steps}(2) one can not replace $\bar R \cup \{1\}$ by $\bar R$. First of all, $\RRaT_\emptyset + \RR\cdot \RRaT_\emptyset \cdot \RR= \RRaT_\emptyset = (\RR,+)^{0}_\emptyset$, where the last equality follows from Corollary~\ref{corollary: finite exponent implies 00=0}. Secondly, $(\RR,+)^{0}_\emptyset \ne \RR^0_\emptyset$, as otherwise $H$ (which is $\emptyset$-definable in the full structure) would contain a finite index ideal. Since $\RR^0_\emptyset = \RR^{00}_\emptyset$ (by Corollary~\ref{corollary: 0=00 in rings}), we get $\RRaT_\emptyset + \RR\cdot \RRaT_\emptyset \cdot \RR \ne \RR^{00}_\emptyset$. So it is enough to prove Lemma~\ref{lemma:nilpotent_RHR}.

	\begin{proof}[Proof of Lemma~\ref{lemma:nilpotent_RHR}]
		Let $\Stab(H)\coloneqq \{ r \in R : r\cdot H \subseteq H\}$. By \cite[Lemma 3.13]{GJK22}, it is enough to show that $\Stab(H)$ has infinite index in $(R,+)$ (actually this sufficiency is easy to see: if $H$ contains a finite index ideal, then $\RR^0_\emptyset \subseteq \bar H$, but $\RR^0_\emptyset$ is a two-sided ideal, so $\RR^0_\emptyset \cdot \bar H \subseteq \RR^0_\emptyset \subseteq \bar H$; thus $\Stab(\bar H)$ is of bounded index, and so of finite index as it is definable; and clearly $\Stab(\bar H)$ is the interpretation of $\Stab(H)$ in the monster model).

		In order to see that the index $[R:\Stab(H)]$ is infinite, it is enough to show that each
		sum $\sum_{2\nmid i} a_{i}X_i$
		(where not all $a_{i}$ are $0$) does not belong to $\Stab(H)$.

		Consider any $f =X_{i_0} + \dots + X_{i_m}$, where $1 \leq i_0 < \dots < i_m$ are odd numbers. Let $g\coloneqq X_{2^{i_0}}$. Then $g \in H$, and we will show that $f \cdot g \notin H$. This implies that $f \notin \Stab(H)$, so we will be done.

		We have $f \cdot g = X_{i_0} X_{2^{i_0}} + X_{i_1}X_{2^{i_0}} + \dots + X_{i_m}X_{2^{i_0}}$. For every $j>0$:
		\begin{itemize}
			\item if $i_j \leq 2^{i_0}$, then $h(X_{i_j}X_{2^{i_0}})=0$, because $2^{i_j} \nmid 2^{i_0}$ (as $i_0 <i_j$),
			\item if $i_j > 2^{i_0}$, then $h(X_{i_j}X_{2^{i_0}})=0$, because $2^{2^{i_0}} \nmid i_j$ (as $i_j$ is odd).
		\end{itemize}
		Also, $h(X_{i_0} X_{2^{i_0}})=1$. Therefore, $h(f \cdot g) =1$, so $f \cdot g \notin H$.
	\end{proof}

	\subsection{Example: the rings \texorpdfstring{$\Z_q^{\oplus \omega}$}{Zq⊕ω} and \texorpdfstring{$\Z_q^{\omega}$}{Zqω}}
	To study the next series of examples, we need to prove a lemma from linear algebra. For $G\coloneqq \Z_q^{\oplus N}$
	(i.e.\ the ring of finitely supported functions from $N$ to $\Z_q$, with $N \leq \omega$, $q\in \omega$, where $\Z_0\coloneqq \Z$)
	and any $i<N$, let $G_i$ be the subgroup of all elements with the support contained in $i+1$. For a subgroup $H$ of $G$, put $H_i\coloneqq G_i \cap H$.

	\begin{lemma}\label{lemma: finite index subgroups of free abelian groups}
		For each finite index subgroup $H$ of $G\coloneqq \Z_q^{\oplus N}$ (where $N \leq \omega$ and $q \in \omega$) there are $g_i \in G_i$, $i<N$, such that $g_i(i)$ is not invertible in $\Z_q$ for less than $[G:H]$ coordinates $i<N$, and $H_i= \sum_{j \leq i} \Z_q\cdot g_j$ for all $i<N$, and $H = \sum_{i <N} \Z_q \cdot g_i$. Moreover, if $q=0$, then $g_i(i) \ne 0$ for all $i<N$ (so $H = \bigoplus_{i \in \omega} \mathbb{Z}\cdot g_i$ and $H_i =\bigoplus_{j \leq i} \mathbb{Z}\cdot g_j$).
	\end{lemma}

	\begin{proof}
		We construct $g_i$ by induction on $i$.

		Since $G_i$ is an abelian group generated by $i+1$ elements, $H_i$ is generated by $i+1$ elements as well. In particular, $H_0$ is generated by one element, which we take as our $g_0$. Suppose $g_0, \dots,g_{i-1}$ are constructed. Choose $i+1$ elements generating $H_i$. Write these elements as rows of an $(i+1) \times (i+1)$ matrix $M$.
		Via Gaussian row reduction, transform $M$ to a lower triangular matrix.
		The last row of the resulting matrix yields an element $g_i \in H_i$. The construction is finished. We need to check that the constructed $g_i$, $i < N$, satisfy our requirements.

		The fact that $H_i= \sum_{j \leq i} \mathbb{Z}_q\cdot g_j$ for all $i<N$ and $H = \sum_{i <N} \mathbb{Z}_q\cdot g_i$ is clear.

		Assume that $g_i(i)$ is not invertible; denote it by $a_i$. Let $e_i$ be the element of $G$ given by $e_i(i)=1$ and $e_i(j) =0$ for $j \ne i$. Since $H_i = H_{i-1} + \mathbb{Z}_q g_i \leq G_{i-1} \oplus a_i\mathbb{Z}_qe_i$, if we choose representatives $r_0,\dots,r_{m-1}$ of all
		cosets of $H_{i-1}$ in $G_{i-1}$, then $r_0,\dots,r_{m-1}$ and $e_i$ lie in pairwise distinct cosets of $H_i$ in $G_i$. Hence, $[G_{i}: H_i] > [G_{i-1}: H_{i-1}]$. Therefore, since $[G:H] \geq [G_i:H_i] $ for all $i<N$, there are less than $[G:H]$ $i$'s with $g_i(i)$ not invertible.

		In the case when $q=0$, if $g_i(i) =0$ for some $i<N$, then $H_i =\sum_{j \leq i} \Z \cdot g_j \subseteq G_{i-1}$, so $[G:H] \geq [G_i:H_i] \geq [G_i:G_{i-1}]= \aleph_0$, a contradiction.
	\end{proof}

	\begin{proposition}\label{prop:Z_q_plus}\,
		\begin{enumerate}[topsep=0pt]
			\item Let $q \in \omega$. For $R\coloneqq \Z_q^{\oplus \omega}$, \nref{condition:group_R00}{1} holds. Thus
			(by Proposition~\ref{proposition: equivalence of finite numbers of steps}), for $q >0$, \nref{condition:X}1 holds as well, and for $q=0$ (i.e.\ $R= \Z^{\oplus \omega}$), \nref{condition:X}{2} holds.
			\item Let $q>0$. For $R\coloneqq \Z_q^\omega$, \nref{condition:group_R00}{1} holds. Thus, \nref{condition:X}{1} holds as well. For $R\coloneqq \Z^\omega$, \nref{condition:X}{2} holds.
		\end{enumerate}
	\end{proposition}

	\begin{proof}
		(1) Consider any finite index subgroup $H$ of $G\coloneqq (R,+)$.
		By Lemma~\ref{lemma: finite index subgroups of free abelian groups}, there is $n \in \omega$ and elements $g_i \in G_i$ for $i<\omega$ such that $g_i(i)$ is invertible in $\Z_q$ for all $i\geq n$ and $H = \sum_{i<\omega} \Z_q g_i$; if $q=0$, we also have $g_i(i) \ne 0$ for all $i<\omega$.

		First, consider the case $q>0$. Fix $m>n$. Then $(\Z_q^{n} \times \{1\}^{m-n} \times \{0\}^{\omega \setminus m}) \cap H \ne \emptyset$, and so $\{0\}^n \times \Z_q^{m-n} \times \{0\}^{\omega \setminus m} \subseteq R\cdot H$. Therefore, the finite index ideal $\{0\} \oplus \dots \oplus \{0\} \oplus \Z_q^{\oplus \omega \setminus n}$ of $R$ is contained in $R\cdot H$,
		so we have \nref{condition:group_R00}{1}.

		Now, assume $q=0$. Fix $m>n$. Then for every $i <n$ there exists $a_i \ne 0$ with $|a_i| \leq |g_i(i)|$ and $(a_0,\dots,a_{n-1}, 1,\dots,1,0,0,\dots) \in H$ (where there are $m-n$ ones). Put $a\coloneqq (g_0(0)\dots g_{n-1}(n-1))!$. Then, $(\Z a)^n \times \Z^{m-n} \times \{0\}^{\omega \setminus m} \subseteq R\cdot H$. Therefore, the finite index ideal $(\Z a)^n \oplus \Z^{\oplus \omega \setminus n}$ is contained in $R\cdot H$, so we have proved \nref{condition:group_R00}{1}.

		(2) Consider first the case of $q>0$. We can, of course, assume that $q>1$. Consider any finite index subgroup $H$ of $(R,+)$.

		For a finite set $\bar r=\{r_0,\dots,r_{n-1}\}$ of elements of $R$, let $B_{\bar r}=\{B_0,\dots,B_{m-1}\}$ be the collection of all atoms of the Boolean algebra of subsets of $\omega$ generated by the finite family $\{ r_i^{-1}(j): i<n, j<q\}$. Then put $S_{\bar r}\coloneqq \{\chi_0,\dots,\chi_{m-1}\}$, where $\chi_i \in R$ is the characteristic function of $B_i$. Finally, let $R_{\bar r}$ be the subgroup of $R$ generated by $S_{\bar r}$. Then $R_{\bar r}$ is a finite subring of $R$ containing $\bar r$.

		\begin{clm*}
			For every finite subset $\bar r $ of $R=\Z_q^\omega$ there is an ideal $I_{\bar r}$ of $R_{\bar r}$ of index smaller than $q^{[R:H]}$ which is contained in $R_{\bar r}\cdot(H \cap R_{\bar r})$.
		\end{clm*}

		\begin{clmproof}
			Since $R_{\bar r} = \Z_q \chi_0 \oplus \dots \oplus \Z_q \chi_{m-1}$ is naturally identified with $\Z_q^m$, we can choose $(g_i)_{i<m}$ provided by the conclusion of Lemma~\ref{lemma: finite index subgroups of free abelian groups} (applied to the subgroup $H \cap R_{\bar r}$ of $R_{\bar r}$ and $N\coloneqq m$). Let $Z_{\bar r}\coloneqq \{i<m: g_i(i) \textrm{ is not invertible}\}$. By the choice of the $g_i$'s and the fact that $[R_{\bar r}:R_{\bar r} \cap H] \leq [R:H]$, we have that $|Z_{\bar r}| < [R:H]$.

			Define $I_{\bar r}$ as the subgroup of $(R_{\bar r},+)$ generated by $\{\chi_i: i \notin Z_{\bar r}\}$. It is easy to see that it is an ideal of $R_{\bar r}$, $[R_{\bar r} :I_{\bar r}]< q^{[R:H]}$, and $I_{\bar r} \subseteq R_{\bar r}(H \cap R_{\bar r})$.
		\end{clmproof}

		For any finite $\bar r$ in $R$, put
		\[
		\mathcal{I}_{\bar r} \coloneqq \{ I \unlhd R_{\bar r}: [R_{\bar r} :I] < q^{[R:H]} \textrm{ and } I \subseteq R\cdot H\}.
		\]
		By Claim 1, $\mathcal{I}_{\bar r}$ is finite and non-empty for every finite $\bar r$ in $R$.

		For $\bar r' \subseteq \bar r$, we have $R_{\bar r'} \leq R_{\bar r}$ and we have the restriction map $\pi_{\bar r \bar r'} \colon \mathcal{I}_{\bar r} \to \mathcal{I}_{\bar r'}$ given by $\pi_{\bar r \bar r'}(I)=I \cap R_{\bar r'}$.

		So we can find
		\[
		(J_{\bar r})_{\bar r} \in \lim_{\longleftarrow} \mathcal{I}_{\bar r},
		\]
		and put $I\coloneqq \bigcup_{\bar r} J_{\bar r}$.

		It is straightforward to check that $I \unlhd R$, $[R:I] < q^{[R:H]}$, and $I \subseteq R\cdot H$. Thus, we have proved \nref{condition:group_R00}{1}, and so \nref{condition:X}{1} follows by Proposition~\ref{proposition: equivalence of finite numbers of steps}.

		To show \nref{condition:X}{2} for $R\coloneqq \Z^{\omega}$, note that if $D$ is thick in $R$, then $\{q\}^\omega \in D$ for some $q>0$, so $qR=(\Z \cdot q)^{\omega} \subseteq R\cdot D$.
		On the other hand, since \nref{condition:X}{1} holds for $\Z_q^\omega \cong \Z^\omega/ (\Z \cdot q)^\omega = R/qR$, we have $(\RR/q\RR) \cdot (\RRaT_\emptyset/q\RR) = (\RR/q\RR) \cdot ((\RR,+)/qR))^{00}_\emptyset = (\RR/q\RR)^ {00}_\emptyset = \RR^{00}_\emptyset/q\RR$, so $q\RR + \RR\cdot \RRaT_\emptyset = q\RR + \RR^{00}_\emptyset$. Therefore, we conclude that $\RR^{00}_\emptyset \subseteq q\RR + \RR\cdot \RRaT_\emptyset \subseteq \RR \cdot \bar D + \RR\cdot \RRaT_\emptyset$. Applying it to all thick $D \subseteq R$ with $\RRaT_\emptyset \subseteq \bar D$, we get $\RR^{00}_\emptyset \subseteq \RR\cdot \RRaT_\emptyset + \RR\cdot \RRaT_\emptyset$. Since the opposite inclusion is obvious, we are done.
	\end{proof}

	\begin{question}\label{question: number of steps for products}{\,}
		\begin{enumerate}[topsep=0pt]
			\item Does $\Z^{\oplus \omega}$ satisfy \nref{condition:X}{1}?
			\item Does $\Z^\omega$ satisfy \nref{condition:group_R00}{1} or even \nref{condition:X}{1}?
			(Note that here \nref{condition:X}{1} $\rightarrow$ \nref{condition:group_R00}{1} by Proposition~\ref{proposition: equivalence of finite numbers of steps}.)
		\end{enumerate}
	\end{question}

	A negative answer to Question~\ref{question: number of steps for products}(1) would show that the implication \nref{condition:group_R00}{n} $\rightarrow$ \nref{condition:X}{n+1} in Proposition~\ref{proposition: equivalence of finite numbers of steps}(4) cannot be improved to \nref{condition:group_R00}{n} $\rightarrow$ \nref{condition:X}{n}.

	\subsection{Example: the ring \texorpdfstring{$\Z_2^{\omega}$}{Z2ω}}
	\label{example: Z2omega}
	It is natural to ask whether one could strengthen Theorem~\ref{theorem: main algebraic theorem} to saying that $H+R\cdot H$ is already a left ideal (equivalently, a subgroup). A weaker potential conclusion could be that there exists $n$ such that for every $A$-definable finite index subgroup $H$ of $(R,+)$, the $n$-fold sum set $((R \cup \{1\}) \cdot H)^{+n}$ is a left ideal (equivalently, subgroup)\footnote{Note that $H+(R\cdot H)^{+n}$ is a subgroup for some $n$ depending on $H$ (e.g.\ by Lemma \ref{lemma: generation in finitely many steps by generic}).}.
	The next example shows that this may fail, even for a commutative, unital ring of positive characteristic.

	\begin{proposition}
		\label{proposition: unbounded number of steps}
		For $R\coloneqq \Z_2^\omega$, for every $n$ there exists a finite index subgroup $H$ of $(R,+)$ such that $R\cdot H$ does not generate a group in $n$ steps.
	\end{proposition}

	\begin{proof}
		Recall that $R \cong (P(\omega), \vartriangle, \cap)$, so we can and will work in the latter ring, still denoted by $R$. For $A\subseteq R$, $\langle A \rangle$ will denote the subgroup additively (i.e.\ using $\vartriangle$) generated by $X$.

		Consider any $n \in \omega$. Choose independent subsets $A_0,\dots,A_{n-1}$ of $\{0,\dots,2^n-1\}$. Put $H\coloneqq \langle P(\omega \setminus 2^n) \cup \{A_0,\dots,A_{n-1}\} \rangle$. Clearly $[R:H] < \omega$. It remains to show that the group additively generated by $R\cdot H$ is not generated in $n-1$ steps. It is clear that we can reduce this problem to a ``finite situation'': whenever $A_0,\dots,A_{n-1}$ are independent subsets of a finite set $X$ and $H_n\coloneqq \langle A_0,\dots,A_{n-1} \rangle$, then the group $\langle P(X) \cdot H_n \rangle$ additively generated by $P(X) \cdot H_n$ is not generated in $n-1$ steps. Observe that $A_0 \cup \dots \cup A_{n-1} \in \langle P(X) \cdot H_n \rangle$ (namely, it can be written as $A_0 \vartriangle (X_1 \cap A_1) \vartriangle \dots \vartriangle (X_{n-1} \cap A_{n-1})$, where $X_i\coloneqq X \setminus \bigcup_{j<i} A_j$).

		Thus, it is enough to show the following claim.
		\begin{clm*}
			$A_0 \cup \dots \cup A_{n-1} \notin (P(X) \cdot H_n)^{+(n-1)}$.
		\end{clm*}
		\begin{clmproof}
			We proceed by induction on $n$. The base step $n=1$ is trivial. Let us do the induction step from $n$ to $n+1$. Suppose for a contradiction that $A_0 \cup \dots \cup A_{n}$ can be presented as a sum of $n$ elements of $P(X) \cdot H_{n+1}$, i.e.\
			\begin{equation}
				\tag{$*$}
				\label{eq:uns_*}
				\begin{split}
					&A_0 \cup \dots \cup A_{n} = \\
					&(X_0 \cap (A_0^{\epsilon_{0, 0}} \vartriangle \dots \vartriangle A_{n-1}^{\epsilon_{0, n-1}})) \vartriangle \dots \vartriangle
					(X_{m-1} \cap (A_0^{\epsilon_{m-1, 0}} \vartriangle \dots \vartriangle A_{n-1}^{\epsilon_{m-1, n-1}})) \vartriangle\\
					&(X_m \cap (A_0^{\epsilon_{m, 0}} \vartriangle \dots \vartriangle A_{n-1}^{\epsilon_{m, n-1}} \vartriangle A_n )) \vartriangle \dots \vartriangle
					(X_{n-1} \cap (A_0^{\epsilon_{n-1, 0}} \vartriangle \dots \vartriangle A_{n-1}^{\epsilon_{n-1, n-1}} \vartriangle A_n)),
				\end{split}
			\end{equation}
			for some $X_0,\dots,X_{n-1} \subseteq X$, where $A^{\epsilon_{i,j}} \coloneqq A$ if $\epsilon_{i,j} = 1$ and $A^{\epsilon_{i, j}} \coloneqq \emptyset$ if $\epsilon_{i, j} =0$.
			Put
			\[
			T_i\coloneqq A_0^{\epsilon_{i, 0}} \vartriangle \dots \vartriangle A_{n-1}^{\epsilon_{i, n-1}}
			\]
			for $i<n$.

			By an easy induction on $n$, one can check that for any $\delta_0,\dots,\delta_{n-1} \in \{0,1\}$
			\begin{equation}
				\label{eq:uns_**}
				\tag{$**$}
				\begin{aligned}
					(A_0)_{\delta_0} \cap \dots \cap (A_{n-1})_{\delta_{n-1}} \subseteq T_i
					&\iff \sum \{\delta_j: \epsilon_{i, j} =1\} \textrm{ is odd},\\
					(A_0)_{\delta_0} \cap \dots \cap (A_{n-1})_{\delta_{n-1}} \cap T_i =\emptyset
					&\iff \sum \{\delta_j: \epsilon_{i, j} =1\} \textrm{ is even},
				\end{aligned}
			\end{equation}
			where $(A)_{\delta_j} \coloneqq A$ if $\delta_j =1$, and $(A)_{\delta_j} \coloneqq X \setminus A$ if $\delta_j=0$.

			We will prove now that
			\begin{equation}
				\label{eq:uns_***}
				\tag{$***$}
				T_0^c \cap \dots \cap T_{m-1}^c \cap T_m \cap \dots \cap T_{n-1} \cap A_n \ne \emptyset,
			\end{equation}
			where $T^c\coloneqq X\setminus T$.
			This yields a contradiction with \eqref{eq:uns_*}: any element $a$ in the above intersection belongs to $A_n$ and does not belong to the right hand side of \eqref{eq:uns_*}, since $a\notin T_i\vartriangle A_n$ for $i\geq m$.

			Let $\mathcal{B}$ be the Boolean algebra of subsets of $A_0 \cup \dots \cup A_{n-1}$ generated by the sets $A_0,\dots A_{n-1}$. By the independence of $A_0,\dots,A_n$, the set $A_n$ does not contain any atom of $\mathcal{B}$. On the other hand, each $A_0^{\epsilon_{i, 0}} \vartriangle \dots \vartriangle A_{n-1}^{\epsilon_{i, n-1}}$ is a union of some atoms of $\mathcal{B}$. By these two observations together with \eqref{eq:uns_*}, we get
			\begin{equation}
				\label{eq:uns_*'}
				\tag{$*'$}
				A_0 \cup \dots \cup A_{n-1} = (X_0' \cap T_0) \vartriangle \dots \vartriangle
				(X_{m-1}' \cap T_{m-1}) \vartriangle
				(X_m' \cap T_m) \vartriangle \dots \vartriangle
				(X_{n-1}' \cap T_{n-1}),
			\end{equation}
			for some $X_0',\dots, X_{n-1}'$.

			By \eqref{eq:uns_*'} and the induction hypothesis, we see that \eqref{eq:uns_*'} is a shortest presentation of $A_0 \cup \dots \cup A_{n-1}$ as a union of elements of $P(X) \cdot H_{n-1}$. This easily implies that $T_0,\dots,T_{n-1}$ are $\Z_2$-linearly independent (with $\vartriangle$ as addition), e.g.\ if $T_2=T_0 \vartriangle T_1$, then each atom of $\mathcal{B}$ contained in $T_2$ would be already contained in $T_0$ or in $T_1$, so we could skip $X_2' \cap T_2$ in \eqref{eq:uns_*'} modifying the other $X_i'$'s appropriately, and this would contradict the previous sentence.

			Let $J_i\coloneqq \{ j : \epsilon_{i, j}=1\}$. Consider the following system of equations:
			\[
			\sum_{j \in J_0} \delta_j \equiv 0 \wedge \dots \wedge \sum_{j \in J_{m-1}} \delta_j \equiv 0 \wedge \sum_{j \in J_{m}} \delta_j \equiv 1 \wedge \dots \wedge \sum_{j \in J_{n-1}} \delta_j \equiv 1,
			\]
			where $\equiv$ denotes the congruence modulo $2$.

			Since $T_0,\dots,T_{n-1}$ are $\Z_2$-linearly independent, the matrix of this system is invertible, hence it has a solution $(\delta_0,\dots,\delta_{n-1}) \in \Z_2^n$.
			By \eqref{eq:uns_**}, the atom $(A_0)_{\delta_0} \cap \dots \cap (A_{n-1})_{\delta_{n-1}}$ is contained in $T_0^c \cap \dots \cap T_{m-1}^c \cap T_m \cap \dots \cap T_{n-1}$. Since also $A_n \cap (A_0)_{\delta_0} \cap \dots \cap (A_{n-1})_{\delta_{n-1}} \ne \emptyset$, we get \eqref{eq:uns_***}.
		\end{clmproof}
	\end{proof}

	Finally, we will elaborate on the last example to answer negatively a question from the introduction whether for every ring $R$ and for any $A$-type-definable, bounded index subgroup $H$ of $(\bar R,+)$ the set $(\bar R \cup \{1\}) \cdot H$ generates a group in finitely many steps? As pointed out in the introduction, Theorems~\ref{theorem: Main Theorem} and \ref{theorem: main topological} tell us that for unital rings and for rings of positive characteristic it is true for $H\coloneqq \RRaT_A$ and $H\coloneqq \RRaT_{\topo}$ when $R$ is topological. Our counter-example takes place in $\Z_2^\omega$, which is unital, commutative and of positive characteristic.

	\begin{lemma}\label{lemma: to get lack of generation in fin. many steps}
		Let $R \coloneqq \Z_2^\omega$. Then there are finite index subgroups $H_1 \geq H_2 \geq \dots$ of $(R,+)$ such that for every pair $m \geq n \geq 1$, $(R\cdot H_m)^{+n} \nsubseteq (R\cdot H_n)^{+(n-1)}$.
	\end{lemma}

	\begin{proof}
		By recursion with respect to $n$, we construct independent subsets $A_0^n,\dots, A_{n-1}^n$ of $Y_n\coloneqq \{0,\dots,2^{n+1}-1\}$ as follows.

		Let us start from $A_0^1\coloneqq \{0\}$. Assume that $A_0^n,\dots,A_{n-1}^n$ have been constructed. Choose independent subsets $Z_0,\dots,Z_n$ of $Y_{n+1} \setminus Y_n$ (they exist, as $|Y_{n+1} \setminus Y_n| = 2^{n+1}$). Put $A_i^{n+1}\coloneqq A_i^n \cup Z_i$ for $i \leq n-1$, and let $A_n^{n+1} \coloneqq A_{n-1}^n \cup Z_n$.

		Define $H_n\coloneqq \langle P(\omega \setminus Y_n) \cup \{A_0^n,\dots,A_{n-1}^n\} \rangle$ (with respect to the group operation $\vartriangle$). Then $H_1 \geq H_2 \geq \dots$ are finite index subgroups of $(R,+)$ (which is identified with $(P(\omega), \vartriangle)$).

		By the claim in the proof of Proposition~\ref{proposition: unbounded number of steps}, we know that $A_0^n \cup \dots \cup A_{n-1}^n \notin (R\cdot H_n)^{+(n-1)}$. Since for $m \geq n$ we have $A_i^{m} \cap Y_n \coloneqq A_i^n$ for $i \leq n-1$, we get that $A_0^m \cup \dots \cup A_{n-1}^m \notin (R\cdot H_n)^{+(n-1)}$ (as otherwise $A_0^n \cup \dots \cup A_{n-1}^n = (A_0^m \cup \dots \cup A_{n-1}^m) \cap Y_n \in (R\cdot H_n)^{+(n-1)} \cap Y_n = (P(Y_n) \cdot H_n)^{+(n-1)} \subseteq (R\cdot H_n)^{+(n-1)}$, a contradiction).
		On the other hand (as in Proposition~\ref{proposition: unbounded number of steps}), we have $A_0^m \cup \dots \cup A_{n-1}^m = A_0^m \vartriangle (X_1 \cap A_1^m) \vartriangle \dots \vartriangle (X_{n-1} \cap A_{n-1}^m) \in (R\cdot H_m)^{+n}$, where $X_i \coloneqq \omega \setminus \bigcup_{j<i} A_j^m$. So we have proved that $A_0^m \cup \dots \cup A_{n-1}^m \in (R\cdot H_m)^{+n} \setminus (R\cdot H_n)^{+(n-1)}$.
	\end{proof}

	\begin{corollary}\label{corollary: infinitely many steps in in the power of Z2}
		Work with $R\coloneqq \Z_2^\omega$ (equipped with the full structure). Take the sequence $H_1 \geq H_2 \geq \dots$ of subgroups of $(R,+)$ produced in the last lemma. Let $H \coloneqq \bigcap \bar H_n$. Then $H$ is a $\emptyset$-type-definable, bounded index subgroup of $(\bar R,+)$ such that $\bar R\cdot H$ does not generate a group in finitely many steps.
	\end{corollary}

	\begin{proof}
		Consider any $n \geq 1$. By the last lemma and compactness, $(\bar R\cdot H) ^{+(n+1)} \nsubseteq (\bar R \cdot \bar H_{n+1})^{+n}$. Hence,
		$(\bar R \cdot H) ^{+(n+1)} \supsetneq (\bar R \cdot H)^{+n}$, and so $\bar R \cdot H$ does not generate a group in $n$ steps.
	\end{proof}

	\section{Algebraic restatements and final comments}\label{section: final comments}

	In this final section, we will explain the purely algebraic meaning of the main results of this paper. But first, let us briefly summarize the remaining questions and the answers obtained in this paper.

	The main remaining questions are:
	\begin{itemize}
		\item Question~\ref{question: introduction, finite number of steps for all rings}, which asks whether for every ring $R$ condition \nref{condition:X}{n} holds for some $n$, and if yes, whether there exists $n$ good for all rings;
		\item Question~\ref{question: 2 steps}, which asks whether \nref{condition:X}{2} or even \nref{condition:X}{1\frac12} holds for all s-unital rings.
	\end{itemize}

	Corollary~\ref{corollary: 3 steps} yields a positive answer to the first question with $n=2\frac12$ for all s-unital rings, and with $n=1\frac12$ for all rings of positive characteristic. Example~\ref{example:independent_functionals} shows that
	we cannot reduce it to $1$ step (not even for s-unital, commutative rings of positive characteristic).
	Theorem~\ref{theorem: bound for fin gen. rings} yields a positive answer to Question~\ref{question: introduction, finite number of steps for all rings} for finitely generated rings
	(albeit with no uniform bound on $n$).

	The list of conditions in Definition~\ref{definition:conditions} is not complete in some sense (which we discuss below), and condition \nref{condition:def_R00}{n} does not
	fit this list very well, as it is not equivalent to \nref{condition:X}{n}; we used \nref{condition:def_R00}{n} as a convenient bridge between \nref{condition:group_R00}{n} and \nref{condition:X}{n}.

	When passing from $(\RR,+)^{0}_A$ to $\RRaT_A$, condition \nref{condition:X}{n} corresponds to the equivalent conditions \nref{condition:X'_subgroup}{n} and \nref{condition:X'_R0}{n}.
	However, what should be the condition corresponding to \nref{condition:group_R00}{n} (having in mind that \nref{condition:group_R00}{n} is essentially an algebraic restatement of \nref{condition:X'_R0}{n})? This is certainly not exactly \nref{condition:def_R00}{n}.
	The condition would be as follows:
	for every $A$-definable (thick) subset $D$ of $R$ such that $\bar D \supseteq \RRaT_A$ there is a descending sequence $(\tilde D_k)_{k \in \omega}$ of $A$-definable, generic and symmetric [or thick] subsets of
	the set $D_i^{+n}$ (where recall that $i \in \{ \leftarrow, \rightarrow, \leftrightarrow\}$) such that
	for all $k \in \omega$ we have:
	\begin{enumerate}
		\item $\tilde D_{k+1} + \tilde D_{k+1} \subseteq \tilde D_k$, and
		\item $R\cdot \tilde D_{k+1} \subseteq \tilde D_k$.
	\end{enumerate}
	By compactness, it follows easily that this condition is indeed
	equivalent to \nref{condition:X}{n}. But how to express it without
	referring to the monster model and $\RRaT_A$?

	In the general case, it seems that we can apply results from Sections 4 and 5 of \cite{CPT22} (with some minor modifications) to express this condition in terms of the so-called ``definable $\delta$-approximate $(\epsilon, \mathbb T^n)$-Bohr neighborhoods in $(R,+)$'' (where $\mathbb T^n$ is the $n$-dimensional torus), but it is rather technical and we are not going to do that here.

	In the special case when we have the full structure on $R$ (so in a purely algebraic context), this can be done using \cite[Proposition 4.16]{GJK22} in the following way:
	one should require the existence of a sequence $(\tilde{D}_k)_{k \in \omega}$ as above for any $D$ which is
	a Bohr neighborhood in $(R,+)$, i.e.\
	an intersection of finitely many sets of the form $\chi^{-1}[(-\frac1m,\frac1m)]$, where $\chi \colon (R,+) \to S^1$ is a character and $m$ is a positive integer (where $S^1$ is identified with the interval $[-\frac12,\frac12)$ with addition modulo $1$).
	Also, as we are working with the full structure, we can, of course, erase the requirement that the $\tilde{D}_k$'s should be $A$-definable. A descending sequence $(\tilde D_k)_{k \in \omega}$, of generic, symmetric subsets of $(R,+)$ satisfying (1) and (2) above could be called a \emph{(weak) ring version of a Bourgain system}. So we see that \nref{condition:X}{n} is equivalent to the existence of such a ring version of a Bourgain system in $D_i^{+n}$ for any Bohr neighborhood $D$ in $(R,+)$. And this is exactly the algebraic meaning of a potential positive answer to Question~\ref{question: introduction, finite number of steps for all rings}, and, in particular, of the answers to this question obtained in Corollary~\ref{corollary: 3 steps} (or Corollary~\ref{corollary: which conditions hold for rings}) and in Theorem~\ref{theorem: bound for fin gen. rings}. In the case when $\RR^{00}_A=\RR^0_A$, so e.g.\ in Corollary~\ref{corollary: 3 steps}, this simplifies as follows:
	the ring Bourgain system $(\tilde D_k)_{k \in \omega}$ can be replaced by a genuine ideal of finite index and contained in $D_i^{+n}$.

	Overall, the above algebraic restatement of \nref{condition:X}{n} seems much more complicated than condition \nref{condition:group_R00}{n} which we were able to completely understand in this paper, showing that it holds for all rings with $n=1\frac12$ and providing examples where it fails with $n=1$.

	If we allow changing $n$, the above requirement $\bar D \supseteq \RRaT_A$ can be dropped, which we now explain. So for any $i \in \{\leftarrow, \rightarrow, \leftrightarrow\}$ and positive half-integer $n$, consider the following condition.

	\begin{enumerate}[label=\textrm{(ii)}$'_n$, ref = \textrm{(ii)$'$}]
		\item
		\label{condition:def_R00'}
		For every generic and symmetric [or thick], $A$-definable subset $D$ of $R$, there is a descending sequence $(\tilde D_k)_{k \in \omega}$ of $A$-definable, generic and symmetric [or thick] subsets of $D_i^{+n}$ such that for all $k \in \omega$ we have:
		\begin{enumerate}
			\item $\tilde D_{k+1} + \tilde D_{k+1} \subseteq \tilde D_k$, and
			\item $R\cdot \tilde D_{k+1} \subseteq \tilde D_k$.
		\end{enumerate}
	\end{enumerate}

	\begin{proposition}
		For any positive integer $n$, \nref{condition:def_R00'}{n} $\rightarrow$ \nref{condition:X}{n} and \nref{condition:X}{n} $\rightarrow$ \nref{condition:def_R00'}{8n}. If $A$ is a model, then $8n$ can be replaced by $4n$. In particular, if $R$ is equipped with the full structure, then $8n$ can be replaced by $4n$.\footnote{We leave stating an analog of this proposition for a positive half-integer but not integer $n$ as an easy exercise.}
	\end{proposition}

	\begin{proof}
		The implication \nref{condition:def_R00'}{n} $\rightarrow$ \nref{condition:X}{n} follows easily by compactness. To show \nref{condition:X}{n} $\rightarrow$ \nref{condition:def_R00'}{8n}, consider any $D$ as in \nref{condition:def_R00'}{8n}. As explained in the proof of Fact \ref{fact: G00=G000}, $\bar D^{+8} \supseteq \RRaT_A$ [and $\bar D^{+4} \supseteq \RRaT_A$ when $A$ is a model].
		So, by \nref{condition:X}{n} and compactness, the desired sequence $(\tilde D_k)_{k \in \omega}$ exists in $(D^{+8})_i ^{+n}$ [in $(D^{+4})_i^{+n}$ when $A$ is a model]. We finish using the fact that $(D^{+m})_i ^{+n} \subseteq D_i^{+mn}$.
	\end{proof}

	By the last proposition, Corollary~\ref{corollary: 3 steps} and Theorem~\ref{theorem: bound for fin gen. rings} yield condition \nref{condition:def_R00'}{n} for the appropriate $n$ (for the rings in question).
	This means that for those rings we have proved that for every $A$-definable, generic and symmetric [or thick] subset $D$ of $(R,+)$ there exists an \emph{$A$-definable} ring Bourgain system in $D_i^{+n}$ for a sufficiently large $n$ (roughly $8$ or $4$ times larger than the one obtained in Corollary~\ref{corollary: 3 steps} and in Theorem~\ref{theorem: bound for fin gen. rings}). And again, whenever $\RR^{00}_A=\RR^0_A$, so e.g.\ in Corollary~\ref{corollary: 3 steps}, the $A$-definable ring Bourgain system $(\tilde D_k)_{k \in \omega}$ can be replaced by a genuine $A$-definable ideal of finite index and contained in $D_i^{+n}$. In the purely algebraic setting, we can drop all the above definability requirements (and the number $n$ is roughly 4 times larger than the one obtained in Corollary~\ref{corollary: 3 steps} and in Theorem~\ref{theorem: bound for fin gen. rings}).

	\section*{Acknowledgments}
	We thank the anonymous referees for helpful remarks.

	\printbibliography

\end{document}